%% file: mcw.tex
\documentclass[3p]{elsarticle}

\makeatletter
\def\ps@pprintTitle{%
 \let\@oddhead\@empty
 \let\@evenhead\@empty
 \def\@oddfoot{}%
 \let\@evenfoot\@oddfoot}
\makeatother

\usepackage{lineno,hyperref}
\modulolinenumbers[5]

\journal{}

\usepackage{graphicx}

\usepackage{amsmath,amssymb,amsthm}

\usepackage{enumerate,enumitem}

\usepackage[dvipsnames]{xcolor}
\usepackage{graphics}

\usepackage{array}
\usepackage{etoolbox}

\usepackage{float}
\usepackage{appendix}

\usepackage{bm}

\usepackage{breqn}

\input{macros.tex}

\providecommand{\Lk}{{\mathrm{L}_K}}
\providecommand{\Th}{{\mathrm{T}}}
\providecommand{\Themp}{{ \widehat{\mathrm{T}} }}
\providecommand{\Inc}{{\mathrm{S}}}
\providecommand{\hInc}{{\widehat{\mathrm{S}}}}
\providecommand{\prefilter}{{G_j}}
\providecommand{\filter}{{F_j}}
\providecommand*{\textN}[1]{\|{#1}\|}

\newcommand{\R}{{\mathbb{R}}}
\newcommand{\C}{{\mathbb{C}}}

\newcommand{\xx}{{\mathcal{X}}}
\newcommand{\hh}{{\mathcal{H}}}

\newcommand{\cG}{{\mathcal{G}}}
\newcommand{\cV}{{\mathcal{V}}}
\newcommand{\cE}{{\mathcal{E}}}

\newcommand{\al}{{\alpha}}

\newcommand{\ka}{{\kappa}}
\newcommand{\la}{{\lambda}}
\newcommand{\eps}{{\epsilon}}

\newcommand{\hpsi}{{\widehat{\psi}}}
\newcommand{\hphi}{{\widehat{\bm{\varphi}}}}

\newcommand{\hla}{{\widehat{\la}}}
\newcommand{\hv}{{\widehat{v}}}

\newcommand{\hK}{{\widehat{\VK}}}
\newcommand{\hu}{{\widehat{\Vu}}}

\DeclareMathOperator{\spn}{span}

\providecommand{\PW}{\mathbf{PW}}
\providecommand{\dom}{\operatorname{dom}}
\newcommand{\HS}{{\textnormal{HS}}}

\newcommand{\ra}[1]{\renewcommand{\arraystretch}{#1}}

\theoremstyle{plain}
	\newtheorem{thm}{Theorem}[section]
	\newtheorem{prop}[thm]{Proposition}
	\newtheorem{lem}[thm]{Lemma}
	\newtheorem{cor}[thm]{Corollary}
\theoremstyle{remark}
	\newtheorem{rmk}[thm]{Remark}
\theoremstyle{definition}
	\newtheorem{dfn}[thm]{Definition}
	\newtheorem{ex}[thm]{Example}
	
\allowdisplaybreaks

\begin{document}

\begin{frontmatter}

\title{Construction and Monte Carlo estimation of\\wavelet frames generated by a reproducing kernel}

\author[DIMA]{Ernesto De Vito}
\ead{devito@dima.unige.it}

\author[ZK]{Zeljko Kereta}
\ead{zeljko@simula.no}

\author[VN]{Valeriya Naumova}
\ead{valeriya@simula.no}

\author[DIBRIS,MIT]{Lorenzo Rosasco}
\ead{lorenzo.rosasco@unige.it}

\author[DIBRIS]{Stefano Vigogna}
\ead{vigogna@dibris.unige.it}

\address[DIMA]{MaLGa - DIMA - Universit\`a di Genova,
Via Dodecaneso 35, 16146 Genova, Italy}

\address[ZK]{Simula Research Laboratory,
Martin Linges Vei 25, 1364 Oslo, Norway}

\address[VN]{Machine Intelligence Department -
SimulaMet,
Pilestredet 52, 0167 Oslo, Norway}

\address[DIBRIS]{MaLGa - DIBRIS - Universit\`a di Genova,
Via Dodecaneso 35, 16146 Genova, Italy}

\address[MIT]{Center for Brains, Minds and Machines - MIT and Istituto Italiano di Tecnologia}

\begin{abstract}
We introduce a construction of multiscale tight frames on general domains.
The frame elements are obtained by spectral filtering of the integral operator associated with a reproducing kernel.
Our construction extends classical wavelets
as well as generalized wavelets on both continuous and discrete non-Euclidean structures
such as Riemannian manifolds and weighted graphs.
Moreover, it allows to study the relation between continuous and discrete frames
in a random sampling regime,
where discrete frames can be seen as Monte Carlo estimates of the continuous ones.
Pairing spectral regularization with learning theory,
we show that a sample frame tends to its population counterpart,
and derive explicit finite-sample rates on spaces of Sobolev and Besov regularity.
Our results prove the stability of frames constructed on empirical data,
in the sense that all stochastic discretizations have the same underlying limit
regardless of the set of initial training samples.
\end{abstract}


\end{frontmatter}

\section{Introduction} \label{sec:intro}

Wavelet systems have long been employed in time-frequency analysis and approximation theory
to break the uncertainty principle
and resolve local singularities against global smoothness.
Nonlinear approximation over redundant families of localized waveforms
has enabled the construction of efficient sparse representations,
becoming common practice in signal processing, source coding, noise reduction, and beyond.
Sparse dictionaries are also an important tool in machine learning,
where the extraction of few relevant features can significantly enhance a variety of learning tasks,
making them scale with enormous quantities of data.
However, the role of wavelets in machine learning is still unclear,
and the impact they had in signal processing has, by far, not been matched.
One objective constraint to a direct application of classical wavelet techniques to modern data science is of geometric nature:
real data are typically high-dimensional and inherently structured,
often featuring or hiding non-Euclidean topologies.
On the other hand, a representation built on empirical samples poses an additional problem of stability,
accounted for by how well it generalizes to future data.
In this paper, expanding upon the ideas outlined in \cite{9030825},
we introduce a data-driven construction of wavelet frames
on non-Euclidean domains,
and provide stability results in high probability.

Starting from Haar's seminal work \cite{Haar1910} and since the founding contributions of Grossmann and Morlet \cite{GM84},
a general theory of wavelet transforms and a wealth of specific families of wavelets have rapidly arisen
\cite{chui92,daubechies1992ten,f2005abstract,mallat1999wavelet,mey92}, first and foremost on $\bbR^d$,
but soon thereafter also on non-Euclidean structures such as manifolds and graphs \cite{CM06,TKP12,DONG2017452,fefupe16,GP11,GBvL17,HVG11,MR3642989}.
Generalized wavelets usually consist of frames
with some kind of broad to tighter link to ideas from multi-resolution analysis.
At the very least, elements of a wavelet frame ought to be associated with locations and scales,
decomposing signals into a sum of local features in increasing resolution.
On a basic conceptual level, many of these generalized constructions
stem from a reinterpretation of the frequency domain as the spectrum of a differential operator.
Indeed, wavelets on $\R$ are commonly generated by dilating and translating a well-localized function $\psi$,
$$
 \psi_{a,b}(x) = |a|^{-1/2} \psi \left( \tfrac{x-b}{a} \right) \qquad a \ne 0 , b \in \R ;
$$
but taking the Fourier transform, they can be rewritten as
\begin{equation} \label{eq:wavelet-mercer}
 \psi_{a,b}(x) = \int |a|^{1/2} \widehat{\psi}(a\xi) e^{2\pi \imath (x-b) \xi} d\xi
 = \int G_a(\xi) \overline{v_\xi(b)} {v_\xi} (x) d\xi ,
\end{equation}
with $  G_a(\xi) = |a|^{1/2} \widehat{\psi}(a\xi) $ and $ v_{\xi} (x) = e^{2\pi \imath x\xi} $.
This allows to reinterpret the wavelet $ \psi_{a,b}(x) $  as a superposition of Fourier harmonics $v_{\xi} (x)$, modulated by a spectral filter $G_a(\xi)$.
Moreover, each $v_{\xi}$ can be seen as an eigenfunction of the Laplacian $ \Delta = -d^2/dx^2 $.
Hence, in principle, we may retrace an analogous construction whenever some notion of Laplacian is at hand.
In particular, Riemannian manifolds and weighted graphs are examples of spaces where this is possible, using the Laplace--Beltrami operator or the graph Laplacian.
A more detailed overview of related work based on these or similar ideas is postponed to Sections \ref{sec:graphs} and \ref{sec:comparison}.

Thus far, the study of generalized wavelets on non-Euclidean domains
has primarily focused on either the continuous \emph{or} the discrete setting.
It is nonetheless natural to investigate the relationship between the two.
For instance, regarding a graph as a sample of a manifold,
we may ask whether and in what sense
the frame built on the graph tends to the one on the manifold.
In this paper we present a unified framework for the construction and the comparison of continuous and discrete frames.
Returning for a moment to the real line,
let us consider the semigroup $ e^{-t \Delta} $ generated by the Laplacian.
This defines an integral operator
$$
 e^{-t \Delta}f(x) = \int K_t(x,y) f(y) dy ,
$$
with $K_t(x,y)$ being the heat kernel.
Such a representation suggests that the generalized Fourier analysis, already revisited as spectral analysis of the Laplacian, can now be translated in terms of a corresponding integral operator (see {\it e.g.} \cite{TKP12,MM08}).
With the attention shifting from the Laplacian to an integral kernel,
our idea is to recast the above constructions inside a reproducing kernel Hilbert space.
Exploiting the reproducing kernel, we will extend a discrete frame out of the given samples,
and thus compare it to its natural continuous counterpart.

Our construction yields empirical frames \smash{$ \widehat\VPsi^N $} on sets of $N$ data.
We will show that \smash{$ \widehat\VPsi^N $} converges in high probability to a continuous frame $\VPsi$ associated to a reproducing kernel Hilbert space $\hh$ as $ N \to \infty $,
thus providing a proof of its stability in an asymptotic sense.
The empirical frames \smash{$ \widehat\VPsi^N $} can be seen as Monte Carlo estimates of $\VPsi$.
Repeated random sampling will in fact produce a sequence of frames \smash{$ \widehat\VPsi^N $} on an increasing chain of finite dimensional reproducing kernel Hilbert spaces \smash{$ \widehat{\hh}_N $}
$$
 \begin{matrix}
  \widehat{\hh}_N & \subset & \widehat{\hh}_{N+1} & \subset & \cdots & \subset & \hh \\[5pt]
  \widehat\VPsi^N & & \widehat\VPsi^{N+1} & & \longrightarrow & & \VPsi 
  \end{matrix} \quad ,
$$
which approximates $\VPsi$ on $\hh$ up to a desired sampling resolution quantifiable by finite sample bounds in high probability.

One may also look at our result as a form of stochastic discretization of continuous frames.
Going from the continuum to the discrete setting is an important problem in frame theory and applications of coherent states.
Given a continuous frame of a Hilbert space, the discretization problem \cite[Chapter 17]{ali2013coherent} asks to extract a discrete frame out of it.
Originally motivated by the need of numerical implementations of coherent states arising in quantum mechanics \cite{DGM,1926NW.....14..664S},
the problem was then generalized to continuous frames \cite{ALI19931} and addressed in several theoretical efforts \cite{FR05,FG07,Gr1991}, until it found a complete yet not constructive characterization in \cite{FS17}.
Sampling the continuous frame is tantamount to sampling the parameter space on which the frame is indexed.
For a wavelet frame, this means the selection of a discrete set of scales and locations.
While the discretization of the scales can be readily obtained by a dyadic parametrization,
the difficult part is usually sampling locations, that is, the domain where the frame is defined.
How to do this is known in many cases and consists in an attentive selection of nets of well covering but sufficiently separated points.
Already sensitive in the Euclidean setting, this procedure can be hard to generalize and implement in more general geometries \cite{TKP12}.
In this respect, our Monte Carlo frame estimation provides a randomized approach to frame discretization as opposed to a deterministic sampling design.
Clearly, our Monte Carlo estimate is not solving the discretization problem in its original form,
since it defines frames only on finite dimensional subspaces.
It is rather providing an asymptotic approximate solution, computing frames on an invading sequence of subspaces $ \widehat{\hh}_N \subset \hh $.
We should also remark that, due to covering properties, standard frame discretization always entails a loosening of the frame bounds;
hence, in particular, only non-tight frames may be sampled, even when the starting continuous frame is Parseval.
As a result, signal reconstruction with respect to the discretized frame will in general require the computation of a dual frame, which is a problem on its own.
On the contrary, in our randomized construction we preserve the tightness,
albeit at the expense of a (possibly large) loss of resolution power $ \hh \setminus \widehat{\hh}_N $.

The remainder of the paper is organized as follows.
The general notation used throughout the paper is listed in Table \ref{tab:notation}.
In Section \ref{sec:graphs} we relate our main contribution
to recent constructions of wavelets on graphs.
This is both a special case and a main motivation of the general theory developed in the subsequent sections.
In Section \ref{sec:RKHS} we introduce the general framework and define the fundamental objects used in our analysis.
The focus is on kernels, reproducing kernel Hilbert spaces, and associated integral operators.
In Section \ref{sec:Frames} we present our frame construction
based on spectral calculus of the integral operator.
Our theory encompasses continuous and discrete frames within a unified formalism,
paving the way for a principled comparison of the two.
In particular, in Section \ref{sec:mcw}, interpreting discrete locations as samples from a probability distribution
we propose a Monte Carlo method for the estimation of continuous frames.
In Section \ref{sec:comparison} we  compare and contrast our approach to the existing literature.
In Section \ref{sec:stability} we prove the consistency of our Monte Carlo wavelets
and obtain explicit convergence rates under Sobolev regularity of the signals.
This is done combining techniques borrowed from the theory of spectral regularization with bounds of concentration of measure.
In Section \ref{sec:Besov} we study the convergence rates in Besov spaces.
In Section \ref{sec:conclusions} we draw our conclusions and point at some directions for future work.

\begin{table}[h]

\caption{Notation} \label{tab:notation}

\centering

\ra{1.5}

\resizebox{\columnwidth}{!}{

\begin{tabular}{l l l l}

\textsf{symbol} & \textsf{definition} & \textsf{symbol} & \textsf{definition} \\

\hline

$\langle \cdot , \cdot \rangle_\CH$, $\N{\cdot}_\CH$ & inner product and norm in a RKHS $\CH$
&
$\Pmat_S$ & orthogonal projection onto a closed subspace $S$ \\

  $\sigma(\Amat)$ & spectrum of a linear operator $\Amat$
  &
  $\supp(\rho)$ & support of a measure $\rho$ \\

$\|\cdot\|$ & operator norm
&
$\langle \cdot , \cdot \rangle_{\rho}$, $\N{\cdot}_{\rho}$ & inner product and norm in $\Ltwo$ \\

$\|\cdot\|_{\text{HS}}$ & Hilbert-Schmidt norm
&
$\delta_x$ & Dirac measure at $x$ \\

$ v \otimes w $ & the operator $ u \in \hh \mapsto \langle u , v \rangle_\hh w \in \hh $
&
$\Vv[i]$ & $i$-th component of a vector $\Vv$ \\

$\Span\{S\}$ & linear span of a set $S$
&
$\VM[i,j]$ & $(i,j)$-th entry of a matrix $\VM$ \\

$S^\perp$ & orthogonal complement of a set $S$
&
$\VM^+$ & pseudoinverse of a matrix $\VM$ \\

$\overline{S}$ & topological closure of a set $S$
&
$X \lesssim Y$ & $ X \le C Y $ for some constant $C>0$ \\

$ S_1 \oplus S_2 $ & direct sum of two subspaces $S_1$ and $S_2$
&
$X \asymp Y$ & $ X \lesssim Y $ and $ Y \lesssim X $ \\

\hline

\end{tabular}

}

\end{table}

\section{Wavelets on graphs and their stability} \label{sec:graphs}

In this section we discuss how the framework introduced in the paper
may be used to study the stability of typical constructions of wavelets on graphs.
We first recall a few elementary concepts about graphs
and set up some notation.
After that, we outline a natural construction of wavelets
based on the graph Laplacian,
and observe that such a construction may be recast in terms of a reproducing kernel.
Finally, we explain how this allows to establish the stability of wavelet frames
in a suitable random graph model.

\subsection{Wavelets on graphs}

We start with some basics of spectral graph theory.
We only review what is strictly necessary for our purposes,
and refer to \cite{Chung1997} for further details.
\begin{dfn}[weighted graph]
An \emph{undirected graph} is a pair $ \cG = ( \cV , \cE ) $,
where $ \cV $ is a finite discrete set of \emph{vertices} $ \cV := \{ x_1 , \dots , x_N \} $,
and $ \cE $ is a set of unordered pairs $ \cE \subset \{ \{ x_i , x_k \} : x_i , x_k \in \cV \} $, called \emph{edges}.
A \emph{weighted} (undirected) graph is an undirected graph with an associated \emph{weight function}
$ w : \cE \to (0,+\infty) $.
\end{dfn}

Arguably, one of the most remarkable facts about graphs is that
it is possible to define on such a minimal structure a consistent notion of Laplacian.
Functions on the graph, more precisely functions $ f : \cV \to \R $,
can be identified with vectors $ \Vf \in \R^N $ by $ \Vf_i := f(x_i) $,
and equipped with the standard inner product $ \Vf^\top \Vg $ for $ \Vf , \Vg \in \R^N $.
As an operator acting on functions,
the graph Laplacian is thus defined by a matrix $ \VL \in \R^{N \times N} $.
\begin{dfn}[graph Laplacian]
Let $ \cG = ( \cV , \cE , w ) $ be a weighted graph.
The \emph{weight matrix} $ \VW := [ w_{i,k} ]_{i,k=1}^N $
is defined by $ w_{i,k} := w(\{x_i,x_k\}) $ for $ \{ x_i , x_k \} \in \cE $, and $ w_{i,k} := 0 $ otherwise.
The \emph{degree matrix} $ \VD := \diag(d_1,\dots,d_N) $ is defined by
$ d_i := \sum_{k=1}^N w_{i,k} $.
The \emph{unnormalized graph Laplacian} is the matrix
$$
 \VL := \VD - \VW .
$$
\end{dfn}
Assuming that $ \cG $ is connected, hence $ d_i > 0 $ for all $ i = 1 , \dots , N $,
the \emph{symmetric normalized} graph Laplacian is
$
 \VL' := \VD^{-1/2} \VL \VD^{-1/2} = \VI - \VD^{-1/2} \VW \VD^{-1/2} .
$
Several other variants are considered in the literature,
including the \emph{random walk} normalized graph Laplacian
$ \VD^{-1} \VL = \VI - \VD^{-1} \VW $, which is not symmetric but conjugate to $\VL'$.
The operators $\VL$, $\VL'$ and further normalizations
result from different definitions of Hilbert structures
on the spaces of functions on $\cV$ and $\cE$ \cite{3213}.
While each operator gives rise to a different analysis,
the choice of one or the other does not have formal consequences in our construction,
hence, for simplicity, we will generically use $\VL$.

The matrix $\VL$ is positive semi-definite,
hence it admits an orthonormal basis of eigenvectors
with non-negative eigenvalues,
customarily sorted in increasing order:
$$
 \VL \Vu_i = \xi_i \Vu_i , \quad i = 0 , \dots , N-1 , \qquad 0 = \xi_0 \le \xi_1 \le \cdots \le \xi_{N-1} .
$$
The spectrum of $\VL$ reveals several important topological properties of the graph.
In particular, a graph has as many connected components as zero eigenvalues,
with eigenfunctions being piecewise constant on the components.
We assume from now on that the graph is connected, hence $ \xi_1 > 0 $.

The graph Laplacian can be seen as a discrete analog of the continuous Laplace operator.
This  analogy justifies
the interpretation of the eigenvectors $ \Vu_i $ as Fourier harmonics,
and the corresponding eigenvalues $ \xi_i $ as frequencies.
Accordingly, the \emph{graph Fourier transform} is defined by
 $$
  \VF := [ \Vu_1 \cdots \Vu_N ]^\top , \qquad [\VF \Vf]_i := \Vu_i^\top \Vf .
 $$
 Note that the indexing is hiding that $ \VF \Vf $ should be thought as a function on the frequencies $ \xi_i $.
 Carrying the analogy forward, a family of \emph{graph wavelets}
can be constructed by spectral filtering of the Fourier basis as follows.
Let $ \{ H_j \}_{j\ge0} $ be a family of functions $ H_j : [0,+\infty) \to [0,+\infty) $ satisfying
\begin{align*}
 & \sum_{j\ge0} H_j(\xi)^2 = 1 \quad \text{for all } \xi \in [0,+\infty) , \\
 & \# \{ H_j : H_j(\xi_i) \ne 0 \} < \infty \quad \text{for } i  = 1 , \dots , N .
\end{align*}
Then, the family
\begin{equation} \label{eq:GW}
 \varphi_{j,k} := \sum_{i=1}^N H_j(\xi_i) \Vu_i[k] \Vu_i \qquad j\geq0,\, k = 1,\dots,N
\end{equation}
defines a Parseval frame on $\cG$ \cite[Theorem 2]{GBvL17}.

Let $ \hh_\cG := \spn\{\Vu_0\}^\perp = \spn\{\Vu_1,\dots,\Vu_{N-1}\} $
the space of all non-constant signals on $\cG$.
The graph Laplacian defines an inner product on $ \hh_\cG $ by $ \langle \Vf , \Vg \rangle_\cG := \Vf^\top \VL \Vg $,
which is invariant under graph isomorphisms.
The Hilbert space $ \hh_\cG $ has \emph{reproducing kernel}
$$
 \VK := \VL^+ .
$$
The matrix $\VK$ on $\hh_\cG$ has same eigenvectors $ \Vu_1 , \dots , \Vu_{N-1} $ as $\VL$, and eigenvalues
$$
 \la_1 = \xi_1^{-1} \ge \la_2 = \xi_2^{-1} \ge \cdots \ge \la_{N-1} = \xi_{N-1}^{-1} .
$$
Therefore, wavelets \eqref{eq:GW} can be as well defined starting from the spectral decomposition of the reproducing kernel $\VK$,
rather than the Laplacian $\VL$.
Conversely, given any reproducing kernel $\VK$,
a frame may be constructed,
without any reference to a Laplacian matrix.
Indeed, this is the point of view taken in this paper.

Besides the equivalence in defining the frame,
starting from a kernel implies some technical differences,
but also opens to new theoretical potential.
First, note that the spectrum gets flipped,
hence the eigenvalues of the kernel should be thought as inverses of Fourier frequencies.
This seemingly irrelevant remark is actually important to correctly interpret
the definitions of Sobolev and Besov spaces given in Section \ref{sec:Besov}.
Moreover, in light of this,
the scale $\tau$ in \eqref{eq:hftauN} can be understood as a frequency threshold,
and the regularization $\tau^{-1}$ in the regression problem \eqref{eq:RLS}
as keeping the low frequencies.
Reasoning in reproducing kernel Hilbert spaces also suggests
further definitions of filtering beyond typical band-pass of Example \ref{ex:classical_filters},
employing regularization techniques from inverse problems, as exemplified in Table \ref{tab:list_of_gs}.
Lastly, reproducing kernels naturally extend the wavelet functions out of the graph vertices,
making possible to analyze the stability of the graph wavelet frame
for different random realizations of the graph.
We elaborate on this in the next section.

\subsection{Stability of wavelets on random graphs}

By virtue of their generality, graphs can be used to model a variety of discrete objects with pairwise relations,
as well as to approximate complex geometries in continuous domains.
In both cases, complexity and uncertainty are often handled by assuming an underlying random model
and studying statistics and asymptotic behavior of relevant variables.
In particular, neighborhood graphs are often used to approximate the Riemannian structure of a manifold.
In a neighborhood graph, vertices are sampled at random from the manifold,
and edges are drawn connecting vertices in suitable neighborhoods,
such as $k$-nearest neighborhoods or $\eps$-radius balls in the ambient Euclidean distance,
or even putting weights using a global (possibly truncated) kernel function.

The convergence of the graph Laplacian to the Laplace--Beltrami operator
has been studied and quantified in several settings,
both as a pointwise \cite{BELKIN20081289,3213,SINGER2006128,MR2387773,10.5555/3104322.3104459} and as a spectral limit \cite{RBD10,vonlux2008,NIPS2006_5848ad95,singer2017,MR4130541}.
On the other hand, wavelets have been generalized to continuous non-Euclidean domains, notably Riemannian manifolds and spaces of homogenous type \cite{TKP12,fefupe16,GP11},
and while the conceptual ingredients remain similar,
the convergence of graph to manifold wavelets is hardly studied.
We next describe how our theory provides a way to fill this gap.

Suppose we have a graph $\cG$ with vertices $\{x_1,\dots,x_N\}$ and a positive definite kernel matrix $\hK$.
For instance, the matrix $N\hK$ may be the kernel associated with the graph Laplacian.
Computing the eigenvalues $\hla_i$ and eigenvectors $\hu_i$ of $\hK$,
we can define, in analogy with \eqref{eq:GW}, the family
\begin{equation} \label{eq:frame-graph}
\hphi_{j,k} := \sum_{i=1}^N F_j(\hla_i) {{\hu_i[k]}} \hu_i \qquad j\geq0 , k = 1,\dots,N,
\end{equation}
for a suitable spectral filter $F_j(\la)$.
By Proposition \ref{prop:frame}, \eqref{eq:frame-graph} defines a Parseval frame on $\cG$.
Now, suppose that the vertices of our graph are sampled from a space $\xx$
with probability distribution $\rho$ and reproducing kernel $K$
satisfying the assumptions of Sections \ref{sec:RKHS} and \ref{sec:Frames}.
Furthermore, suppose that the kernel matrix $\hK$ is given by
$$
 \hK[i,k] = N^{-1} K(x_i,x_k) .
$$
For example, the space $\xx$ may be a compact Riemannian manifold,
in which case we could consider the heat kernel associated with the Laplace--Beltrami operator,
and regard the kernel matrix as a discretization of the integral operator.
As a discrete example,
one may also think of $\xx$ as a supergraph of $\cG$.
Thanks to Proposition \ref{prop:frame}, the family of Monte Carlo wavelets
\begin{equation*}
 \hpsi_{j,k} (x) := \sum_i G_j(\hla_i) \overline{\hv_i(x_k)} \hv_i(x) \qquad j\geq0,\, k = 1,\dots,N
\end{equation*}
is a Parseval frame isomorphic to \eqref{eq:frame-graph}.
Crucially, in this new representation, the frame functions are well-defined both on and off the graph $\cG$,
and thus the convergence of the frame can be studied on a test signal $ f : \xx \to \R $,
as discussed in Section \ref{sec:stability}.
The stability of the graph wavelets \eqref{eq:frame-graph} can therefore be established by an application of Theorem \ref{thm:error} or \ref{thm:Besov_cons}.

Starting from the next section,
we develop our theory in greater generality,
but always bearing in mind the motivating setting just discussed.

\section{Preliminaries} \label{sec:RKHS}

In this section we prepare the technical ground on which our results will built (see also \cite{RBD10}).
Let $\CX$ be a locally compact, second countable topological space endowed with a Borel probability measure $\rho$.
Given a continuous, positive semi-definite kernel
$$
K : \CX \times \CX \to \C ,
$$
we denote the associated reproducing kernel Hilbert space (RKHS) by
\begin{equation*}
 \CH 
 := \overline{\spn} \{ K_x : x\in\CX\} ,
\end{equation*}
where $ K_x := K(\cdot,x) \in \CH$, and the closure is taken with respect to the inner product $ \langle K_x , K_y \rangle_{\CH} := K(y,x)$.
Elements of $\CH$ are continuous functions satisfying the following reproducing property:
\begin{equation}\label{eqn:reproducing_property}
f(x) = \EUSP{f}{K_x}_\CH \quad \text{for all } f \in \CH.
\end{equation}
The space $\CH$ is separable, since $\CX$ is separable.
We further assume $K$ is bounded on $\CX$ and denote
\begin{equation*}
\kappa := \sup_{x \in \CX} \sqrt{K(x,x)} = \sup_{x \in \CX} \|
K_x\|_\CH  < \infty ,
\end{equation*}
which implies that $\CH$ is continuously embedded into the space of bounded continuous functions  on $\CX$.

We define the (non-centered) covariance operator $\Th : \CH \to \CH $ by
\begin{equation}\label{eqn:Th_defn}
\Th := \int_\CX K_x \otimes K_x \,d\rho(x) ,
\end{equation}
where the integral converges strongly.
The operator $\Th$ is positive and trace-class (therefore compact) with $ \sigma(\Th) \subset [0,\kappa^2] $.
Hence, the spectral theorem ensures the existence of a countable orthonormal set $ \{ v_i\}_{i \in \CI_\rho \cup \CI_0} \subset \CH $
and a sequence $ ( \la_i )_{i\in\CI_\rho} \subset (0,\kappa^2] $
such that
\begin{equation*}
 \Th v_i = \begin{cases}
              \lambda_i v_i & i\in \CI_\rho \\
              \hspace{1pt} 0 & i\in \CI_0
             \end{cases} .
\end{equation*}

Let  $\Ltwo$ be the space of square-integrable functions on $\CX$ with respect to the measure $\rho$,
and denote $ \CX_\rho := \supp(\rho)$.
We define the integral operator $\Lk: \Ltwo\rightarrow \Ltwo$ by
$$
\Lk F (x) := \int_\CX K(x,y) F(y)\, d\rho(y) .
$$ 
The spaces $\CH$ and $\Ltwo$ and the operators $\Th$ and $\Lk$ are related through the inclusion operator $\Inc: \CH \rightarrow \Ltwo$ defined by
$$
\Inc f(x) := \EUSP{f}{K_x}_{\CH} .
$$ 
The adjoint operator $\Inc^*:  \Ltwo\rightarrow  \CH$ acts as the strongly converging integral
$$
\Inc^* F= \int_\CX F(x) K_x\,d\rho(x) .
$$
We have $ \Th = \Inc^* \Inc $ and $ \Lk = \Inc \Inc^* $.
Hence, $\sigma(\Th)\backslash\{0\}=\sigma(\Lk)\backslash\{0\}$,
and the eigenfunctions $ \{ u_i \}_{i\in\CI_\rho\cup\CI_0} \subset \Ltwo $ of $\Lk$ satisfy
\begin{equation} \label{eq:svd}
 \Inc v_i = \begin{cases}
                  \sqrt{\lambda_i} u_i & i \in \CI_\rho \\
                  0 & i\in\CI_0
                 \end{cases} .
\end{equation}
Mercer's theorem gives
\begin{equation} \label{eq:Mercer}
\begin{aligned}
& K(x,y) = \sum_{i \in \CI_\rho \cup \CI_0} \overline{v_i(x)} v_i(y) \quad \text{for } x,y \in \CX , \\
& K(x,y) = \sum_{i \in \CI_\rho} \la_i \overline{u_i(x)} u_i(y) \quad \text{for } x,y \in \CX_\rho ,
\end{aligned}
\end{equation}
where the series converge absolutely and uniformly on compact subsets.

Defining
\begin{equation*}
 \CH_\rho := \overline{\spn} \{ K_x : x \in \CX_\rho\} = \overline{\spn} \{ v_i : i \in \CI_\rho \} ,
\end{equation*}
where the closure is taken in $\hh$,
we can identify $\CH_\rho$ as a (non-closed) subspace of $\Ltwo$.
The closure of $\CH_\rho$ in $\Ltwo$ is
$$
\overline{\CH}_\rho : = \overline{\spn} \{ u_i : i \in \CI_\rho \} ,
$$
and the following decompositions hold true:
\[
\CH = \CH_\rho \oplus \ker \Inc , \qquad \Ltwo =
  \overline{\CH}_\rho \oplus \ker \Inc^\ast .
 \]
For $f \in \CH_\rho$, we can relate the norms in $\CH$ and $\Ltwo$  as
\begin{equation} \label{eq:sqrt(T)f}
\textN{f}_{\rho} = \| \sqrt \Th f \|_\CH .
\end{equation}
In other words, $\sqrt{\Th}$ induces an isometric isomorphism between $\overline{\CH}_\rho$ and ${\CH}_\rho$.
We define the partial isometry $\Umat:\CH\to\Ltwo$, such that $\Umat \CH_\rho = \overline{\CH}_\rho$, by
\begin{equation*}
\Umat f = \sum_{i\in\CI_\rho} \EUSP{f}{v_i}_{\CH} u_i .
\end{equation*}

As examples of this setting, we may think of $\CX$ as $\bbR^d$, or a non-Euclidean domain such as a compact connected Riemannian manifold or a weighted graph.
In these cases, we can take $K$ as the heat kernel associated with
the proper notion of Laplacian, be it the Laplace--Beltrami operator or
the graph Laplacian.

\section{Wavelet frames by reproducing kernels} \label{sec:Frames}

We now build Parseval frames in the RKHS $\CH$ and in $\Ltwo$.
Our construction is centered around eigenfunctions of the integral operator \eqref{eqn:Th_defn} and filters on the corresponding eigenvalues.
Continuous frames emerged in the mathematical physics community from the study of coherent states,
as a generalization of the more common notion of a discrete frame \cite{ali2013coherent,f2005abstract}.

\begin{dfn}[frame]
Let $\CH$ be a Hilbert space, $\CA$ a locally compact space and $\mu$ a Radon measure on $\CA$ with $\supp \mu = \CA$.
A family $\VPsi =\{ \psi_a: a\in \CA\}\subset \CH$ is called a \emph{frame} for $\CH$ if there exist constants $0<A\leq B<\infty$ such that, for every $f\in\CH$, we have
\[ A \N{f}_\CH^2 \leq \int_{\CA} \SN{\EUSP{f}{\psi_a}_\CH}^2 d\mu(a) \leq B \N{f}_\CH^2.\]
We say that $\VPsi$ is \emph{tight} if $A=B$, and \emph{Parseval} if $A=B=1$. 
\end{dfn}
In the above definition it is implicitly assumed that the map $a \mapsto \EUSP{\Psi_a}{f}_\hh$ is measurable for all $f\in\hh$.
It is important to note that this definition depends on the choice of the measure $\mu$. 
In the case of a counting measure,
we recover the standard definition of discrete frame.

\subsection{Filters} \label{sec:filters}

To construct our wavelet frames, we first need to define filters,
{\it i.e.}
functions acting on the spectrum of $\Th$ that satisfy a partition of unity condition.
\begin{dfn}[filters]\label{defn:filter}
A family $\{\prefilter\}_{j\geq0} $ of measurable functions $\prefilter : \break [0,+\infty) \to [0,+\infty)$  such that
    \begin{equation}\label{eqn:part_unity}
      \lambda\sum_{j\geq0} \prefilter(\lambda)^2  = 1 \quad \text{for all } \lambda\in(0,\kappa^2]
    \end{equation}
    is called a family of \emph{filters}.
\end{dfn}
\noindent By the spectral theorem, $\prefilter(\Th)$ is a
(possibly unbounded) positive operator on $\CH$ such that \(\sigma(\prefilter(\Th))= \prefilter(\sigma(\Th)),\)
with domain of definition 
\begin{equation*}
 \CD_j := \Big\{f\in\CH : \sum_{i\in \CI_\rho \cup \CI_0}  \prefilter(\lambda_i)^2 \,\SN{\EUSP{f}{v_i}_{\CH}}^2 <\infty\Big\} .
\end{equation*}
It follows that
\begin{equation*}
\CD:= \Span\{ v_i:i\in \CI_\rho \cup \CI_0\}\subset \CD_j \quad \text{for all } j \ge 0 ,
\end{equation*}
and
\begin{equation*}
\prefilter (\Th) v_i =
\begin{cases}
  \prefilter (\lambda_i)  v_i, & i\in \CI_\rho \\
  \prefilter (0) v_i,  &  i\in \CI_0
\end{cases}.
\end{equation*}

An easy way to define filters is by differences of suitable spectral functions.
\begin{dfn}[spectral functions]
A family
$\{g_j\}_{j\geq0} $  of measurable functions
$g_j : [0,\infty) \to [0,\infty)$ satisfying
\begin{equation} \label{eqn:g}
 0 \le g_j \le g_{j+1} , \qquad \lim_{j\to\infty} \lambda g_j(\lambda) = 1 \quad \text{for all } \lambda\in(0,\kappa^2]
\end{equation}
is called a family of \emph{spectral functions}.
\end{dfn}
Given a family of spectral functions $\{g_j\}_{j\geq0} $, filters $\{\prefilter\}_{j\geq0} $ can be obtained setting
\begin{equation} \label{eqn:Gg}
 G_0(\lambda) := \sqrt{g_0(\lambda)} , \qquad G_{j+1} (\lambda) := \sqrt{ g_{j+1}(\lambda) - g_j(\lambda) } \quad \text{for } j\geq 0.
\end{equation}
The filters thus defined give rise to a telescopic sum:
\begin{equation} \label{eq:sumG=g}
\sum_{j\leq \tau} \prefilter(\lambda)^2=g_\tau(\lambda) .
\end{equation}
Taking the limit for $\tau\rightarrow\infty$, condition \eqref{eqn:part_unity} is satisfied thanks to \eqref{eqn:g}. 
Conversely, starting from a family of filters $\{G_j\}_{j\geq0}$, we can define spectral functions $\{g_j\}_{j\geq0} $ by
$$
g_j(\lambda):=\sum_{\ell\le j} G_\ell(\lambda)^2 \quad \text{for } j \ge 0 ,
$$
which enjoys \eqref{eqn:g} due to \eqref{eqn:part_unity}.
Therefore, the notion of filter and that of spectral function are equivalent,
and we will refer to them interchangeably.

The definition in \eqref{eqn:Gg} allows to find a wealth of filters by tapping
into regularization theory \cite{enhane96}.
In the forthcoming analysis, we will use the following notion of qualification.
\begin{dfn}[qualification]
The \emph{qualification} of a spectral function $g_j:[0,\infty)\rightarrow[0,\infty)$
is the maximum constant $ \nu \in (0,\infty] $ such that
\begin{equation*} 
\sup_{\lambda\in(0,\kappa^2]} \lambda^\nu\SN{1-\lambda g_j(\lambda)}\leq C_\nu j^{-\nu} \quad \text{for all } j \ge 0 ,
\end{equation*}
where the constant $C_\nu$ does not depend on $j$.
\end{dfn}
In the theory of regularization of ill-posed inverse problems \cite{enhane96},
the qualification represents the limit within which a regularizer may exploit the regularity of the true solution.
In particular, methods with finite qualification suffer from the so-called \emph{saturation effect}.

Some standard examples of spectral functions, together  with their qualifications, are listed in Table \ref{tab:list_of_gs}.

\begin{table}[H]
\caption{Spectral regularizers and their qualifications.
Landweber iteration 
and Nesterov acceleration 
require $ \gamma < 1 / \kappa^2 $ and $ \beta \geq 1 $.
 In heavy ball, $ \alpha_j,\, \beta_j $ are suitably selected sequences depending on $\nu$,
 where $\nu$ is any positive real
 (see \cite{Pagliana2019ImplicitRO}).
}
\label{tab:list_of_gs}
\centering
\ra{2}
\begin{tabular}{l c c}
\multicolumn{1}{c}{\textsf{method}} & $g_j(\lambda)$ & \textsf{qualification} \\
\hline
{Tikhonov regularization} & $\dfrac{1}{\lambda + 1/j} $ & $1$ \\
{iterated Tikhonov} ($m$ iterations) & $ \dfrac{ (\lambda + 1/j)^m - (1/j)^m }{\lambda(\lambda + 1/j)^m } $ & $m$ \\
{Landweber iteration} & $ \frac{1}{\lambda} ( 1 - (1 - \gamma \lambda)^{j} ) $ & $\infty$ \\
{asymptotic regularization} & $ \frac{1}{\lambda} ( 1 - \exp(- j\lambda) ) $ & $\infty$ \\
 {heavy ball ($\nu$-method)} &  $(1-\alpha_{j}\lambda+\beta_{j})g_{j-1}(\lambda)-\beta_{j}g_{j-2}(\lambda)+\alpha_{j}$ & $\nu$\\
 {Nesterov acceleration} & $ (1-\gamma\lambda)\big(g_{j-1}(\lambda)+\frac{j-2}{j-1+\beta}(g_{j-1}(\lambda)-g_{j-2}(\lambda)\big)+\gamma$ & $\nu\geq1/2$\\
\hline
\end{tabular}
\end{table}

Additional examples of admissible filters widely used in the construction of wavelet frames (see \emph{e.g.} \cite{TKP12,fefupe16})
are given by the following:
\begin{ex}[localized filters] \label{ex:classical_filters}
Let $ g \in C^\infty([0,\infty)) $
such that $ \supp(g) \subset (2^{-1},\infty) $, $ 0 \le g \le 1 $, and $ g(\la) = 1 $ for all $ \la \ge 1 $.
Define
$$
\la g_{j}(\la) := g(2^j\la) .
$$
Then the family $ \{ g_j \}_{j\ge 0} $ satisfies the properties \eqref{eqn:g}.
Furthermore, the corresponding filters \eqref{eqn:Gg} are localized, meaning that,
defining $ F_j(\la) := \sqrt{\la} G_j(\la) $, we have
$$
 \supp( F_0 ) \subset (2^{-1},\infty) , \qquad  \supp( F_j ) \subset ( 2^{-j-1} , 2^{-j+1} ) \quad \text{for } j \ge 1 .
$$
\end{ex}

\subsection{Frames} \label{sec:frames}

We are now ready to define our wavelet frames.
{We first form frame elements in $\CH$, and then use the partial isometry $\Umat:\CH\to\Ltwo$ to obtain frames in $\Ltwo$.}
\begin{dfn}[wavelets]\label{defn:frames}
Let $\{\prefilter\}_{j\geq0}$ be a family of filters as in Definition \ref{defn:filter},
and assume
\begin{equation} \label{eq:KxDj}
 K_x \in \CD_{j} \quad \text{for all } j \geq 0 \text{ and almost every } x\in \CX_{\rho} .
\end{equation}
We define the families of \emph{wavelets}
$$
\VPsi:=\{\psi_{j,x}: j\geq0,x\in\CX_\rho\}\subset \CH , \qquad
\VPhi:=\{\varphi_{j,x}: j\geq0,x\in\CX_\rho\}\subset \Ltwo ,
$$
where
\begin{align}\label{eqn:frame_defn}
  \psi_{j,x} :=\prefilter(\Th)K_x , \qquad
  \varphi_{j,x} := \Umat \prefilter(\Th)K_x  \quad \text{for } j\geq0 \text{ and } x\in \CX_\rho. 
 \end{align}
\end{dfn}
Observe that, since $ \psi_{j,x} $ and $ \varphi_{j,x} $ are defined for $ x \in \CX_\rho $,
we actually have $ \VPsi \subset \CH_\rho \subset \CH $,
and  $ \VPhi \subset \overline{\CH}_\rho \subset \Ltwo $.
In particular, the orthogonality of $\CH_\rho$ and $\ker\Inc$ entails $ \EUSP{K_x}{G_j(T)v_i}_\CH = 0 $ for all $ i \in \CI_0 $.
By the reproducing property \eqref{eqn:reproducing_property},
condition \eqref{eq:KxDj} is thus equivalent to  
\begin{equation*} 
\sum_{i\in \CI_\rho} \prefilter(\lambda_i)^2\SN{ v_i(x) }^2 <\infty \quad 
  \text{for all }  j\geq0 \text { and almost every } x\in \CX_{\rho}.
\end{equation*}
If $\prefilter$ is a bounded function, then $\prefilter(\Th)$ is a bounded operator, hence $\CD_j=\CH$.
In this case, which includes the spectral functions listed in Table
\ref{tab:list_of_gs}, condition \eqref{eq:KxDj} is trivially satisfied.

Using the spectral decomposition of $\prefilter(\Th)$
and the reproducing property,
we obtain
\begin{equation} \label{eqn:psikx_series}
\psi_{j,x}(y) = \sum_{i\in \CI_\rho} \prefilter(\lambda_i) \overline{v_i(x)} v_i(y) , \quad
\varphi_{j,x}(y) = \sum_{i\in \CI_\rho} \sqrt{\lambda_i}\prefilter(\lambda_i) \overline{u_i(x)}\, u_i(y) .
\end{equation}
These expressions allow to interpret $\VPsi$ and $\VPhi$ as families of wavelets, in the sense of \eqref{eq:wavelet-mercer}.
We interpret $x$ as the location and $j$ as the scale parameter;
the functions $K_x$ localize the signal in space,
whereas the filters $\prefilter$ regularize or localize in frequency.
Note also the analogy with \eqref{eq:Mercer},
in the light of which \eqref{eqn:psikx_series} may be seen as a filtered Mercer representation.

With the following proposition we show that \eqref{eqn:frame_defn} defines Parseval frames.
\begin{prop}\label{prop:frame}
Assume the setting in Section \ref{sec:RKHS},
and let $ \VPsi , \VPhi $ be defined as in Definition \ref{defn:frames}.
Then, for every $f\in \CH$ we have
\begin{equation}\label{eqn:frame_prop_H}
  \sum_{j\geq0}\int_{\CX}  \SN{\EUSP{f}{\psi_{j,x}}_{\CH}}^2\,
    d\rho(x) =  \big\|{\Pmat_{\CH_\rho}f}\big\|_\CH^2,
\end{equation}
and for any $F\in \Ltwo$  we have
\begin{equation} \label{eqn:frame_prop_L2}
  \sum_{j\geq0}\int_{\CX}  \big\lvert{\EUSP{F}{\varphi_{j,x}}_\rho}\big\rvert^2\,
    d\rho(x) =  \big\|{\Pmat_{\overline{\CH}_\rho} F}\big\|_{\rho}^2.
\end{equation}
\end{prop}
\begin{proof}
The equality \eqref{eqn:frame_prop_L2} follows from \eqref{eqn:frame_prop_H}
and the fact that $U$ is unitary from $\CH_\rho$ to $\overline{\CH}_\rho$.
To establish \eqref{eqn:frame_prop_H},
in view of Lemma~\ref{dense}
it suffices to consider functions in the dense subspace $ \CD \subset \CH $.
Thus, let $ f \in \CD $.
Since $\prefilter(\Th)$ is self-adjoint on $\CD_j$, and $ \CD \subset \CD_j $ for all $j$, we have
\[
\EUSP{f}{\psi_{j,x}}_\CH = \EUSP{f}{\prefilter(\Th) K_x}_{\CH} = \EUSP{\prefilter(\Th) f}{K_x}_{\CH} ,
\]
which integrated over $x\in\CX$ gives
\begin{equation} \label{eq:Fj(T)f}
\int_{ \CX}  \SN{\EUSP{f}{\psi_{j,x}}_{\CH}}^2\, d\rho(x) =
\EUSP{\Th \prefilter(\Th) f }{\prefilter(\Th) f}_\CH =
\EUSP{\Th \prefilter(\Th)^2\,f}{f}_\CH .
\end{equation} 
Summing over $j\geq0$ and using \eqref{eqn:part_unity}, we therefore obtain
\begin{align*}
\sum_{j\geq0} \EUSP{\Th\prefilter(\Th)^2\, f}{f}_{\CH}
& = \sum_{i\in\CI_\rho} \Big(\SN{\EUSP{f}{v_i}_{\CH}}^2\sum_{j\geq0} \lambda_i\prefilter(\lambda_i)^2 \Big) \\
& = \sum_{i\in \CI_\rho} \SN{\EUSP{f}{v_i}_{\CH}}^2
= \N{\Pmat_{\CH_\rho}f}_{\CH}^2.
\end{align*}
\qed
\end{proof}

The frame property can also be expressed as a resolution of the identity.
Such a formulation will be particularly useful in Section \ref{sec:stability}.

\begin{prop} \label{prop:resolution}
Under the assumptions of Proposition \ref{prop:frame},
there exists a positive bounded operator $\Th_j:\CH\rightarrow\CH$ such that
\begin{equation} \label{eq:Tj}
\Th_j = \int_\xx \psi_{j,x}\otimes \psi_{j,x} \, d\rho(x) ,
\end{equation}
where the integral converges weakly.
Furthermore,
\begin{align}
      & \Th_j= \Th \prefilter(\Th)^2 ,
      \label{eq:4} \\
      & \sum_{j\le\tau} \Th_j = \Th g_\tau(\Th) ,
      \label{eq:5}
    \end{align}
and the following resolution of the identity holds true:
\begin{equation} \label{eq:resolution}
 \Pmat_{\CH_\rho} = \sum_{j\ge0} \Th_j .
\end{equation}
\end{prop}
\begin{proof}
 From \eqref{eq:Fj(T)f} we have,
 for all $ f \in \CD $,
 $$
 \int_{ \CX}  \SN{\EUSP{f}{\psi_{j,x}}_{\CH}}^2\, d\rho(x) \le \|\Th \prefilter(\Th)^2\| \|f\|_{\hh}^2 ,
 $$
where $\Th \prefilter(\Th)^2$ is bounded since $\lambda \prefilter(\lambda)^2 \le 1$ by \eqref{eqn:part_unity}.
Hence, thanks to Lemma~\ref{dense}, there exists a positive bounded operator
$\Th_j$ as in \eqref{eq:Tj}.
Moreover, \eqref{eq:Fj(T)f} implies \eqref{eq:4} by the density of $\CD$.
The equality \eqref{eq:5} follows from \eqref{eq:4} and \eqref{eq:sumG=g}.
Lastly, \eqref{eq:resolution} is a reformulation of \eqref{eqn:frame_prop_H}.
\qed
\end{proof}

{Depending on the choice of the measure $\rho$, Proposition \ref{prop:frame} gives the frame property for either a continuous or a discrete setting.}
Namely, consider a discrete set $ \{x_1,\ldots,x_N\} $, and let
\[\widehat{\rho}_N := \frac{1}{N}\sum_{k=1}^N \delta_{x_k} . \]
With the choice of the discrete measure $ \widehat{\rho}_N$,
\eqref{eqn:Th_defn} defines the discrete (non-centered) covariance operator
$ \Themp:\CH\to\CH $ by
\begin{equation*} 
\Themp := \frac{1}{N} \sum_{k=1}^N K_{x_k}\otimes K_{x_k} .
\end{equation*}
Furthermore, Definition \ref{defn:frames} produces the family of wavelets
\begin{equation*} 
 \widehat{\psi}_{j,k} := \prefilter(\Themp) K_{x_k} \quad \text{for } j\geq0 \text{ and } k=1,\dots,N ,
\end{equation*}
which, by Proposition \ref{prop:frame}, constitutes a discrete Parseval frame on
\begin{equation*} 
 \widehat{\CH}_N := \CH_{\widehat{\rho}_N} = \spn\{ K_{x_k} : k=1,\ldots,N\} \simeq \bbC^N.
\end{equation*}
In Section \ref{sec:mcw} we will make reference to this construction to define Monte Carlo wavelets,
where the points $x_1,\ldots,x_N$ are drawn at random from $\CX_\rho$.

\subsection{Two generalizations}

We discuss here two generalizations of the framework presented in Section \ref{sec:frames}.
First, one may readily consider more general scale parameterizations.
Namely, let $\Omega$ be a locally compact, second countable topological space, endowed with a measure $\mu$ defined on the Borel $\sigma$-algebra of
$\Omega$, finite on compact subsets, and such that $\supp \mu=\Omega$. 
Adjusting the definitions accordingly, such as replacing the sums over all non-negative integers $j$ in \eqref{eqn:part_unity} and \eqref{eqn:frame_prop_H} with integrals over $\Omega$ with respect to $\mu$, {the proof of Proposition \ref{prop:frame} follows along the same steps.}
In this context, Definition \ref{defn:filter} can be seen as a special case where $\Omega$ is countable and $\mu$ is the counting measure.
Second, the assumption that the kernel $K$ is bounded,
  implying that  $\Lk$ admits an orthonormal basis of eigenvectors, is not necessary
  for our construction of Parseval frames.   Indeed, it is enough to assume that
  \[
\int_{ \CX}\SN{ f(x)}^2\, d\rho(x) <+\infty \quad \text{for all } f\in\hh .
    \]
This implies that $\hh$ is a subspace of $L^2(\xx;\rho)$ and the inclusion operator $\Inc$ is bounded.
The integral \eqref{eqn:Th_defn} converges now in the weak operator topology,
and the covariance operator $\Th$ is positive and bounded.
Thus, the Riesz--Markov theorem entails that, for all $f\in\CH$, there is a
unique finite measure $\nu_f$ on $[0,+\infty)$ such that
$\nu_f\LRP{[0,+\infty)}=\N{f}_\CH^2$ and
\begin{equation*}
  \EUSP{\Th f}{f}_\CH = \int_{[0,{+\infty})} \lambda d\nu_f(\lambda).
\end{equation*}
By spectral calculus, there exists a unique positive operator
$\prefilter(\Th):\CD_j\to\CH$ such that
\begin{equation*}
  \EUSP{\prefilter(\Th)f}{f}_\CH= \int_{[0,{+\infty})}  \prefilter(\lambda) d\nu_f(\lambda),
\end{equation*}
where now
\begin{equation*}
  \CD_j :=\Big\{f\in\CH : \int_{[0,{+\infty})}   \prefilter(\lambda)^2 d\nu_f(\lambda) <\infty\Big\}.
\end{equation*}
Assume further that
\[
\CD_\infty :=\{f\in\CH : f\in\operatorname{dom} \prefilter(\Th)^2
  \text{ for all } j\geq0\}
\]
is a dense subset of $\CH$. 
Assumption  \eqref{eq:KxDj} and Definition \ref{defn:frames} are still valid. 
Moreover, the proof of Proposition \ref{prop:frame} remains essentially unchanged. The only difference is in the following lines of equalities: for a given
$f\in\CD_\infty$, we have
\begin{align*}
   \sum_{j\geq0} \EUSP{\prefilter(\Th)^2\Th f}{f}_{\CH} & =\sum_{j\geq0}
    \Big( \int_{[0,{+\infty})}  \lambda\prefilter(\lambda)^2  d\nu_f(\lambda) \Big) \\
       & =   \int_{(0,{+\infty})}  \big(  \sum_{j\geq0} \lambda\prefilter(\lambda)^2
  \big)  d\mu_f(\lambda)\\&= \int_{(0,{+\infty})}  1 \,d\mu_f(\lambda)   = \N{\Pmat_{\CH_\rho} f}_{\CH}^2,
\end{align*}
where the second equality is due to Tonelli's theorem.

\section{Monte Carlo wavelets} \label{sec:mcw}

We are finally ready to define our Monte Carlo wavelets.
In the following, we adopt notations, definitions and assumptions of Sections \ref{sec:RKHS} and \ref{sec:Frames}.
For the sake of simplicity, we further assume  $\supp(\rho) = \CX$, so that $\CH_\rho=\CH$.
By Proposition \ref{prop:frame}, the family $\VPsi$ defined in \eqref{eqn:frame_defn}
describes a Parseval frame on the entire Hilbert space $\CH$.
 
\begin{dfn}[Monte Carlo wavelets] \label{def:mcw}
Suppose we have $N$ independent and identically distributed samples $x_1,\ldots,x_N\sim\rho$.
Consider the empirical covariance operator $ \Themp:\CH\to\CH $ defined by
\begin{equation*} 
\Themp := \frac{1}{N} \sum_{k=1}^N K_{x_k}\otimes K_{x_k} .
\end{equation*}
Let $\{\prefilter\}_{j\geq0}$ be a family of filters as in Definition \ref{defn:filter}.
We call
\begin{equation*} 
\widehat\VPsi^N:=\big\{\widehat{\psi}_{j,k} := \prefilter(\Themp) K_{x_k}\,:\, j\ge0 \text{ and } k=1,\dots, N\big\}
\end{equation*}
a family of \emph{Monte Carlo wavelets}.
\end{dfn}
The family $ \widehat\VPsi^N $ of Definition \ref{def:mcw} corresponds to
the family $\VPsi$ of Definition \ref{defn:frames}
with respect to the empirical measure $ \widehat{\rho}_N := \frac{1}{N}\sum_{k=1}^N \delta_{x_k} $.
Hence, thanks to Proposition \ref{prop:frame}, $\widehat{\VPsi}^N$ defines a discrete Parseval frame on the finite dimensional space
$$
\widehat{\CH}_N := \spn\{ K_{x_k} :k=1,\ldots,N\} .
$$
Now, let $\VPsi$ be the family of wavelets in the sense of Definition \ref{defn:frames} with respect to the (continuous) measure $\rho$.
Again by Proposition \ref{prop:frame}, $\VPsi$ is a (continuous) Parseval frame on the (infinite dimensional) space $\CH$.
Taking more and more samples, we obtain a sequence of frames $\widehat{\VPsi}^N$ on a chain of nested subspaces of increasing dimension:
\[ \widehat{\CH}_N\subset\widehat{\CH}_{N+1}\subset \cdots \subset\CH.\]
We thus interpret $ \widehat{\VPsi}^N $ as a Monte Carlo estimate of $ \VPsi $.
In this view, we are interested in studying the asymptotic behavior of $\widehat{\VPsi}^N$ as $N\rightarrow\infty$, and, in particular, the convergence of $\widehat{\VPsi}^N$ to $\VPsi$.

Notice that, despite being finite-dimensional,
the frame $ \widehat\VPsi^N $ consists of functions that are well-defined on the entire space $\xx$.
In particular, for any signal $f$ in the reproducing kernel Hilbert space $ \hh $,
we can study the wavelet expansion
\begin{equation} \label{eq:wavelet_approx}
 f \approx \sum_{j\le\tau} \sum_{k=1}^N \langle f ,
 \hpsi_{j,k} \rangle_\hh \hpsi_{j,k} . 
\end{equation}
This series approximates $f$ up to a resolution $\tau$ and a sampling rate $N$.
Our main result (Theorem \ref{thm:error}) states that, cutting off the frequencies at a threshold $ \tau = \tau(N) $ and letting $N$ go to infinity,
the error of \eqref{eq:wavelet_approx} goes to zero,
\begin{equation*}
\Big\| f - \sum_{j\le\tau(N)} \sum_{k=1}^N \langle f ,
\hpsi_{j,k} \rangle_\hh \hpsi_{j,k} \Big\|_\hh
\xrightarrow{{N\to\infty}} 0 , 
\end{equation*}
at a rate that depends on the regularity of the signal $f$.
In other words, the frame constructed on the sample space $\{x_1,\dots,x_N\}$ is asymptotically resolving the signal defined on the space $\xx$.
This result will be derived as a finite-sample bound in high probability.

\paragraph{Deterministic discretization vs random sampling}

Discretization is a classical problem in frame theory, harmonic analysis and applied mathematics {\it tout court}.
While the construction of reproducing representations may usefully exploit rich topological, algebraic and measure theoretical properties of a continuous parameter space,
discretization is eventually required when it comes to numerical implementation.
Starting from a continuous frame $ \{ \psi_a : a \in \CA \} $ in a Hilbert space $\hh$, frame discretization selects a countable subset of parameters $ \CA' \subset \CA $
so that the corresponding subfamily $ \{ \psi_a : a \in \CA' \} $ preserves the frame property.
This typically involves a deterioration of the frame bounds,
which grows with the sparsity of $\CA'$.

A possible interpretation of our Monte Carlo wavelets is as a randomized approximate  frame discretization.
Random sampling may be useful when the topology of the parameter space is complex or unknown.
On the other hand, our discrete frame is not a frame on the original space $\hh$,
but only on a finite dimensional approximation $\widehat{\hh}$ of $\hh$.
Notice though that our frame preserves the tightness,
and the signal loss $ \hh \setminus \widehat{\hh} $ is asymptotically zero.
Moreover, the numerical implementation of any discretized frame on $\hh$ would still require truncation at finitely many terms, resulting in fact in a loss of the global frame property.
Lastly, when the space is unknown and we can only access signals trough finite samples,
going beyond the given sampling resolution might per se not be significant,
while our results characterize how the frame parameters may be chosen
adaptively to the given sampling rate.

\paragraph{Numerical implementation}

The representation of $ \widehat{\psi}_{j,k} $ in Definition \ref{def:mcw} is remarkably compact,
but hardly suitable for computation.
We next provide an implementable formula of our Monte Carlo wavelets,
using the Mercer representation \eqref{eqn:psikx_series}
along with the singular value decomposition \eqref{eq:svd}.
Let $ \Themp \hv_i = \hla_i \hv_i $ be the eigendecomposition of $\Themp$.
Then \eqref{eqn:psikx_series} reads as
\begin{equation*} 
 \hpsi_{j,k} (x) = \sum_{i=1}^N G_j(\hla_i) \overline{\hv_i(x_k)} \hv_i(x) \qquad j\geq0,\, k = 1,\dots,N ,
\end{equation*}
where the eigenpairs $ (\hla_i,\hv_i) $ can be computed from the \emph{kernel matrix}
\begin{equation} \label{eq:kernel_matrix}
 \VK[i,k] := K(x_i,x_k) \qquad i , k = 1,\dots,N .
\end{equation}
Indeed, we have $ \Themp = \hInc^* \hInc $ and $ N^{-1} \VK = \hInc \hInc^* $,
where $\hInc$ is the \emph{sampling operator}
\begin{equation} \label{eq:sampling}
 \hInc : \CH \to \C^N ,
 \qquad
 (\hInc f)[i] = f(x_i)
 \qquad
 i = 1,\dots,N ,
\end{equation}
and $\hInc^*$ is the \emph{out-of-sample extension}
\begin{equation} \label{eq:out-of-sample}
 \hInc^* : \C^N \to \CH ,
 \qquad
 (\hInc^*\Vu) (x) = \frac{1}{N} \sum_{\ell=1}^N K(x,x_\ell) \Vu[\ell]
 \qquad
 x \in \xx .
\end{equation}
Thus, the eigenvalues $\hla_i$ of $\Themp$ are exactly the eigenvalues of $ N^{-1}\VK$.
Moreover, in view of \eqref{eq:svd}, the eigenfunctions $\hv_i$
can be obtained from the eigenvectors $\hu_i$ of $ N^{-1}\VK$ by
$$
 \hv_i
 = \hla_i^{-1/2} \hInc^* \hu_i
 = \hla_i^{-1/2} \frac{1}{N} \sum_{\ell=1}^N \hu_i[\ell] K_{x_\ell} ,
$$
which evaluated at $x_k$ gives
$$
 \hv_i(x_k)
 = \hla_i^{-1/2} \frac{1}{N} \sum_{\ell=1}^N K(x_k,x_\ell) \hu_i[\ell]
 = \hla_i^{-1/2} N^{-1} (\VK \hu_i)[k]
 = \hla_i^{1/2} \hu_i[k] .
$$
We therefore obtain the computable formula
\begin{equation*}
\hpsi_{j,k} (x) = \frac{1}{N} \sum_{i,\ell=1}^N G_j(\hla_i) \overline{\hu_i[k]} \hu_i[\ell] K(x,x_\ell) \qquad j\geq0,\, k = 1,\dots,N .
\end{equation*}
For what concerns the Monte Carlo wavelet transform of a signal $ f \in \CH $,
it is easy to see that
$$
 \langle f , \hpsi_{j,k} \rangle_\CH = \VU G_j(\Lambda) \VU^* f(x_k) ,
$$
where $ N^{-1} \VK = \VU \Lambda \VU^* $ expresses the eigendecomposition of $N^{-1}\VK$ in matrix form.

\paragraph{Computational considerations}
The bottleneck in the implementation of our Monte Carlo wavelets is the eigendecomposition of the kernel matrix,
which requires in general $\CO(N^3)$ operations
and is therefore impractical in typical large scale scenarios.
This is in fact a common problem for virtually all spectral based constructions of frames (see {\it e.g.} \cite{GBvL17,HVG11,MM08}).
A possible solution
is approximating the filters by low order polynomials,
thus simplifying the functional calculus to repeated matrix-vector multiplication,
which scales well in the case of sparse graphs
\cite{HVG11}.
While kernel matrices are typically dense,
such an approach may still be useful for compactly supported kernels \cite{wendland95},
although their real applicability is mostly limited to the low-dimensional regime.
Besides sparsity, a more reasonable property to leverage is fast eigenvalue decay,
which opens onto a variety of methods for truncated approximate SVD.
Deterministic methods allow to compute an $r$-rank approximation in $\CO(r N^2)$ \cite{SSB19},
whereas randomized methods can further reduce the complexity to $\CO(\log r N^2 + r^2 N)$ \cite{HMT11,RanLinAlg}.

We also remark that the actual Monte Carlo approximation of a given signal
is in principle a different problem than the computation of the frame itself,
and as such may in some cases be more tractable.
For example, for some specific filters as in Table \ref{tab:list_of_gs},
the computation of \eqref{eq:wavelet_approx} boils down to
the implementation of some regularized inversion or minimization procedure,
for which several approaches based on sketching, random projections, hierarchical decompositions and early stopping may be profitably used \cite{pmlr-v51-camoriano16,JMLR:v18:15-376,NIPS2017_850af92f,NIPS2015_5936,NIPS2017_05546b0e,NIPS2017_61b1fb3f,2017arXiv170602205S,earlystopping}.
An efficient implementation of Monte Carlo wavelets is out of the scope of this paper
and will be subject of future work.

\section{Comparison with other frame constructions} \label{sec:comparison}

The approach we adopt in Section \ref{sec:Frames} differs from the existing literature in several crucial aspects.
We now give an overview of similarities and differences.
As argued in Section \ref{sec:intro}, many techniques for the analysis of signals on non-Euclidean domains, such as manifolds and graphs,
are based on spectral filtering of some suitable operator.
There are, generally speaking, two distinct yet related perspectives.

A first type of methods builds frames for function spaces on compact differentiable manifolds associated with certain positive operators (predominantly the Laplace--Beltrami operator).
In \cite{TKP12,GP11}, filter functions $g_j$ are applied to the given operator $\Lmat$, giving $g_j(\sqrt{\Lmat})$ for $j\geq0$.
One then needs to ensure that this defines an integral operator with a corresponding kernel $\psi_j(\sqrt{\Lmat})(x,y)$, which often poses a technical challenge, and relies on the relationship between the operator $\Lmat$ and local metric properties of the manifold.
We avoid this by using a positive definite kernel from the start.
The next step is to sample points $\{x^j_k\}_{k=1}^{m_j}$ from the manifold for each scale $j$,
in such a way that they form a $\delta_j$-net and satisfy a cubature rule for functions in the desired space.
Frame elements are then defined by $C_{j,k}\,\psi_j(\sqrt{\Lmat})(x^j_k,\cdot)$, for some suitable weights $C_{j,k}$.
The resulting family of functions constitutes a non-tight frame on the entire function space.
On the contrary, our sampled frames are Parseval frames on finite-dimensional subspaces.
As we are going to show in the next section,
in order to establish convergence
we do not require a stringent selection of points; instead, we sample at random, which allows for a straightforward algorithmic approach,
independent of the specific geometry of the underlying space.

In a different line of research \cite{MM08,M10,WANG202064}, frames are built on an arbitrary orthonormal basis $\{w_i\}_{i\geq0}$ of a separable Hilbert space of functions defined on a quasi-metric measure space, together with a suitable sequence of positive reals $(l_i)_{i\geq0}$.
Based on these data, a kernel-like function $K_H(x,\cdot) := \sum_{i\geq0} H(l_i) w_i(x)w_i$ is constructed.
This mirrors the basis expansion of frame elements \eqref{eqn:psikx_series},
but in our case a specific orthonormal basis is taken, that is, the eigenbasis of the integral operator,
and $(l_i)_{i\geq0}$ are the corresponding eigenvalues.
Due to the use of an arbitrary basis and sequence, an additional effort (or a set of assumptions) needs to be made in order to ensure the desired properties, such as the decay of the approximation error as the number of eigenvalues resolved by the function $H$ increases. 
Some of the results are similar to those in our paper, albeit estimation errors or sample bounds have not been established in this context.

On the other hand, starting from a discrete setting,
graph signal processing
considers a weight (or adjacency) matrix to define a certain graph operator $\Lmat$, such as the graph Laplacian \cite{GBvL17,HVG11} or a diffusion operator \cite{CM06}.
The frame elements are then defined in the spectral domain as $\psi_{j,x} := g_j( \Lmat)\delta_x$, where $g$ is an admissible wavelet kernel, $j$ a scale parameter, and $\delta_x$ the indicator function of a vertex $x$. 
This is conceptually similar to \eqref{eqn:frame_defn}, though there are also several distinctions. 
First, following \cite{GBvL17}, our construction results in  \emph{Parseval} frames.
This simplifies the computational effort, since Parseval frames are canonically self-dual,
and thus signal reconstruction does not require the computation of a dual frame.
Moreover, to localize the frame in space we use the continuous kernel function $K_x$, instead of the impulse $\delta_x$.
Since in our setting the kernel $K$ is used both to define the underlying integral operator and to localize the frame elements, we can use the theory of RKHS to establish a connection between continuous and discrete frames, as we will show in Section \ref{sec:stability}.
In typical constructions of frames on graphs, a more judicious effort is usually required to elaborate analogous convergence results.

\section{Stability of Monte Carlo wavelets} \label{sec:stability}

In this section we study the relationship between continuous and discrete frames,
regarding the latter as Monte Carlo estimates of the former.
{We begin by restricting our attention to $\CH$, and we will then extend the analysis to $\Ltwo$.}
Let
\begin{equation*} 
\Th_j := \int_{\CX} \psi_{j,x} \otimes \psi_{j,x} d\rho(x) , \qquad \Themp_j :=\frac{1}{N} \sum_{k=1}^N \widehat{\psi}_{j,k} \otimes \widehat{\psi}_{j,k}
\end{equation*}
be the frame operators associated with the {scale} $j$, and its empirical counterpart.
By Proposition~\ref{prop:resolution}, we have
\[ \Id_\CH = \sum_{j\geq0} \Th_j , \qquad \Id_{\widehat\CH_N} = \sum_{j\geq0} \Themp_j.\]
For $f\in\CH$, given a threshold scale $\tau\in\mathbb N$ and a sample size $N$, we let
\begin{align} \label{eq:hftauN}
 \widehat{f}_{\tau,N}:=\sum_{j=0}^\tau \Themp_j f
\end{align}
be the empirical approximation of $f$ using the first $\tau$ scales of the frame $\widehat{\VPsi}^N$.
The reconstruction error of $\widehat{f}_{\tau,N}$ can be decomposed into
\begin{equation}\label{eqn:error_splitting}
\N{f- \widehat{f}_{\tau,N}}_\CH\leq \Big\|{\sum_{j>\tau} \Th_j f}\Big\|_\CH + \Big\|{\sum_{j=0}^\tau \LRP{\Th_j-\Themp_j}f}\Big\|_\CH.
\end{equation}
The first term is the \emph{approximation error}, {arising from the truncation of the resolution of} the identity. 
The second term is the \emph{estimation error}, which stems from estimating the measure by means of empirical samples.
Next, we derive quantitative error bounds for both terms, and then balance the resolution  $\tau$ in terms of sample size $N$ to obtain our convergence result.

\paragraph{Approximation error}
Note that Proposition \ref{prop:frame} already implies 
$$
\textN{\sum_{j>\tau} \Th_j f}_\CH \xrightarrow{\tau\rightarrow\infty} 0 ,
$$
being the tail of a convergent series.
To quantify the speed of convergence with respect to $\tau$,  approximation theory suggests that  $f$ has to obey some notion of regularity.
In the following we assume a smoothness of Sobolev kind (see \cite{fefupe16} and Section \ref{sec:Besov}),
also known in statistical learning theory as the \emph{source condition} (see \cite{caponnetto-devito}):
\[ f = \Th^\alpha h \text{ for some } h\in\CH \text{ and } \alpha>0.\]
\begin{prop}\label{prop:approx_bound}
Assume that $g_j$ has qualification $ \nu \in (0,\infty] $ and $f\in\range(\Th^\alpha)$ for some $ \alpha > 0 $.
Let $\beta := \min\{\nu, \alpha\}$.
Then
\begin{equation*} 
\Big\|{\sum_{j>\tau} \Th_j f}\Big\|_\CH \lesssim \N{\Th^{-\alpha} f}_\CH \kappa^{2(\alpha-\beta)}\tau^{-\beta}.
\end{equation*}
\end{prop}
\begin{proof}
By~\eqref{eq:5} we have $ \sum_{j>\tau} \Th_j = \Id_\CH - \Th g_\tau(\Th) $.
Hence,
\begin{align*}
\Big\|{\sum_{j>\tau} \Th_j f}\Big\|_\CH^2 &= \sum_{i\in\CI_\rho} | 1-\lambda_i g_\tau(\lambda_i)|^2 \SN{\EUSP{f}{v_i}_\CH}^2 \\
& = \sum_{i\in\CI_\rho} \LRP{\lambda_i^{\beta} | 1-\lambda_i g_\tau(\lambda_i) | }^2 \SN{\EUSP{\Th^{-\beta} f}{v_i}_\CH}^2 \\
& \leq \biggl( \sup_{i\in\CI_\rho} \lambda_i^{\beta} | 1-\lambda_i g_\tau(\lambda_i)| \biggr)^2 \sum_{i\in\CI_\rho} \SN{\EUSP{\Th^{-\beta} f}{v_i}_\CH}^2 \\
& \lesssim  \tau^{-2\beta} \kappa^{4(\alpha-\beta)} \N{\Th^{-\alpha} f}_\CH^2 .
\end{align*}
\qed
\end{proof}

\paragraph{Estimation error}
To bound the second term in \eqref{eqn:error_splitting}, we rely on concentration results for covariance operators \cite{RBD10}.
\begin{prop}\label{prop:estimation_bound}
Assume that $\lambda\mapsto \lambda g_\tau(\lambda)$ is Lipschitz continuous on $[0,\kappa^2]$ with Lipschitz constant $L(\tau)$.
Then, for every $f\in\CH$ and $ t > 0 $, with probability at least $1-2e^{-t}$ we have
\begin{equation*}
\Big\|{\sum_{j=0}^\tau \LRP{\Th_j-\Themp_j}f}\Big\|_\CH \lesssim \N{f}_\CH\kappa^2\sqrt{t} L(\tau) N^{-1/2}.
\end{equation*}
\end{prop}
\begin{proof}
Using~\eqref{eq:5} and Lemma \ref{lem:Lip} we have
\begin{align*}
\Big\|{\sum_{j=0}^\tau \LRP{\Th_j-\Themp_j}f}\Big\|_\CH
& = \Big\| \LRP{\Th g_\tau(\Th) - \Themp g_\tau(\Themp)} f \Big\|_\CH \\
& \le \Big\| \Th g_\tau(\Th) - \Themp g_\tau(\Themp) \Big\|_{\operatorname{HS}}\N{f}_\CH \\
& \leq L(\tau)\big\|{\Th-\Themp}\big\|_{\operatorname{HS}} \N{f}_\CH.
\end{align*}
Bounding $\textN{\Th-\Themp}_{\operatorname{HS}} $ with the concentration estimate \cite[Theorem 7]{RBD10} we obtain
\[ \big\|{\Th-\Themp}\big\|_{\operatorname{HS}} \lesssim \kappa^2 \sqrt{t} N^{-1/2} \]
with probability no lower than $1-2e^{-t}$.
\qed
\end{proof}
All examples of filters given in Section \ref{sec:filters} satisfy the Lipschitz condition required in Proposition \ref{prop:estimation_bound}.
\begin{lem} \label{lem:lip}
Let $ g_j $ be a spectral function from Table \ref{tab:list_of_gs}.
 Then the function $ \la \mapsto \la g_\tau(\la) $ is Lipschitz continuous on $[0,\ka^2]$,
with Lipschitz constant $ L(\tau) \lesssim \tau $
for the first four spectral functions,
and $ L(\tau) \lesssim \tau^2 $ for the last two.
 Moreover, let $ g_j $ be defined as in Example \ref{ex:classical_filters}, with $ |g'| \le B $.
 Then the function $ \la \mapsto \la g_\tau(\la) $ is Lipschitz continuous on $[0,\ka^2]$,
with Lipschitz constant $ L(\tau) \le B 2^\tau $.
\end{lem}
\begin{proof}
For the first four spectral functions of Table \ref{tab:list_of_gs},
the claim follows by bounding the explicit derivative of $ \lambda \mapsto \lambda g_\tau(\lambda)$;
for the last two, from an application of Markov brothers' inequality (see \cite[Supplemental, Lemma 1]{Pagliana2019ImplicitRO}).
For filters of Example \ref{ex:classical_filters}, we differentiate $ \la \mapsto g(2^{\tau}\la) $
and use $ |g'| \le B $.
\qed
\end{proof}

\begin{rmk} \label{rmk:sharper}
In this paper we are not interested in the constants. We rely on the Hilbert norm since it provides both a simple bound on
$\big\|{\Th-\Themp}\big\|_{\operatorname{HS}} $ and,  by the Lipschitz
assumption, the stability bound $\big\| \Th g_\tau(\Th) - \Themp g_\tau(\Themp)
\big\|_{\operatorname{HS}}\N{f}_\CH \leq
L(\tau)\big\|{\Th-\Themp}\big\|_{\operatorname{HS}}  $.  Our result can be
improved by using the sharper bound
\[
  \big\|{\Th-\Themp}\big\| \leq  C \big\|\Th\| \max\Big\{ \sqrt{\frac{{r(\Th)}}{N}}, \frac{r(\Th)}{N},
  \sqrt{\frac{t}{N}}, \frac{r(\Th)}{N}\Big\}, 
\]
where $r(\Th)=\frac{\operatorname{trace}(\Th)}{\N{\Th}}$
(see Theorem~9 in \cite{kolo17} and the techniques in the proof of
Theorem 3.4 in \cite{blmu18} to bound $\big\| \Th g_\tau(\Th) - \Themp g_\tau(\Themp)
\big\|$).
\end{rmk}

\paragraph{Reconstruction error and convergence}
Combining Propositions \ref{prop:approx_bound} and \ref{prop:estimation_bound}, we can finally prove the convergence of our Monte Carlo wavelets.
In order to balance approximation and estimation error, we need to tune the resolution $\tau$ with the number of samples $N$ and the smoothness $\alpha$ of the signal, in so far as the qualification $\nu$ of the filter allows.

\begin{thm} \label{thm:error}
Assume that $g_\tau$ has qualification $ \nu \in (0,\infty] $,
$f\in\range(\Th^\alpha)$ for some $ \alpha > 0 $,
and $\lambda\mapsto \lambda g_\tau(\lambda)$ is Lipschitz continuous on $[0,\kappa^2]$ with Lipschitz constant $L(\tau)\lesssim \tau^p$, $ p \ge 1 $.
Let $\beta:=\min\{\alpha,\nu\}$
and set
$$
\tau := \lceil N^{\frac{1}{2(\beta+p)}}\rceil .
$$
Then, for every $ t > 0 $, with probability at least $1-2e^{-t}$ we have
\begin{equation*}
\big\|{f-\widehat{f}_{\tau,N}}\big\|_\CH \lesssim \big\|{\Th^{-\alpha} f}\big\|_\CH \big({\kappa^{2(\alpha-\beta)}+\kappa^{2\alpha+2}\sqrt{t}}\big) N^{-\frac{\beta}{2(\beta+p)}}.
\end{equation*}
\end{thm}
\begin{proof}
Starting from the decomposition \eqref{eqn:error_splitting},
we bound the two terms by Propositions \ref{prop:approx_bound} and \ref{prop:estimation_bound}.
The approximation error is $\CO(\tau^{-\beta})$, while the estimation error is $\CO(\tau^pN^{-1/2})$.
We thus choose $\tau$ to balance them out, and collect the constants.
\qed
\end{proof}

If $\supp \rho\neq\CX$, we have instead a frame on $\CH_\rho$,
and the corresponding resolution of the identity \(\Id_{\CH_\rho} = \sum_{j\geq0} \Th_j\).
The reconstruction error would thus include an additional bias term:
\begin{equation*}
\big\|{f- \widehat{f}_{\tau,N}}\big\|_\CH \leq \N{\Pmat_{\ker\Inc} f}_\CH + \big\|{\sum_{j>\tau} \Th_j f}\big\|_\CH + \Big\|\sum_{j=0}^\tau \big({\Th_j-\Themp_j}\big)f\Big\|_\CH .
\end{equation*}

Classical spectral functions from Table \ref{tab:list_of_gs} satisfy the assumptions of Theorem \ref{thm:error}.
We report the explicit rates in Table \ref{tab:list_of_error_rates}.
A convergence result for filters of Example \ref{ex:classical_filters} will be provided at the end of Section \ref{sec:Besov}.
 \begin{table}[h]
 \caption{Error rates for signals $ f \in \range(\Th^\alpha) $ and several spectral regularizers.}
 \label{tab:list_of_error_rates}
\centering
\ra{1.5}
\begin{tabular}{l l c l}
\multicolumn{1}{c}{\textsf{method}} & \multicolumn{1}{c}{\textsf{error rate in $\textN{\cdot}_\CH$}} & & \multicolumn{1}{c}{\textsf{error rate in $\textN{\cdot}_\rho$}} \\
\hline
{Tikhonov regularization} & $N^{-\frac{\min\{\alpha, 1\}}{2\min\{\alpha, 1\}+2}}$  & & $N^{-\frac{\min\{\alpha+1/2, 1\}}{2\min\{\alpha+1/2, 1\}+2}}$ \\
{iterated Tikhonov $(m)$} & $N^{-\frac{\min\{\alpha, m\}}{2\min\{\alpha, m\}+2}}$ & & $N^{-\frac{\min\{\alpha+1/2, m\}}{2\min\{\alpha+1/2, m\}+2}}$ \\
{Landweber iteration} & $N^{-\frac{\alpha}{2\alpha+2}}$ & & $N^{-\frac{\alpha+1/2}{2\alpha+3}}$ \\
{asymptotic regularization} & $N^{-\frac{\alpha}{2\alpha+2}}$ & & $N^{-\frac{\alpha+1/2}{2\alpha+3}}$ \\
{heavy ball $(\nu)$} & $N^{-\frac{\min\{\alpha, \nu\}}{2\min\{\alpha, \nu\}+4}}$ & & $N^{-\frac{\min\{\alpha+1/2, \nu\}}{2\min\{\alpha+1/2, \nu\}+4}}$ \\
{Nesterov acceleration} & $N^{-\frac{\min\{\alpha, \nu\ge1/2\}}{2\min\{\alpha, \nu\ge1/2\}+4}}$ & & $N^{-\frac{\min\{\alpha+1/2, \nu\ge1/2\}}{2\min\{\alpha+1/2, \nu\ge1/2\}+4}}$ \\
\hline
\end{tabular}
\end{table}

\paragraph{Convergence in $\Ltwo$}
Error rates in $\Ltwo$ can be extracted using the isometry between $\overline{\CH}_\rho$ and  $\CH_\rho$.
Suppose again for simplicity that $ \supp \rho = \CX $.
In view of \eqref{eq:sqrt(T)f}, for $f\in\CH_\rho=\CH$ we have
\begin{align*} \big\|{f-\widehat{f}_{\tau,N}}\big\|_{\rho}  &= \big\|{\sqrt{\Th}(f-\widehat{f}_{\tau,N})}\big\|_{\CH}.
\end{align*}
{Decomposing the error into its approximation and estimation components, we can repeat the same analysis as in the proof of Theorem \ref{thm:error}. 
The estimation bound simply gets an additional $\kappa$ factor.
Assuming $ f \in \Th^\alpha \CH$ with $\alpha>0$,
for the approximation term we have
\begin{align*}
\big\|{\sqrt{\Th} \sum_{j>\tau}\Th_j f}\big\|_\CH
& \leq \sup_{i\in\CI_\rho} \big({\lambda_i^{\beta}\LRP{1-\lambda_i g_\tau(\lambda_i)}}\big)\sum_{i\in\CI_\rho} \big\lvert{\big\langle\Th^{1/2-\beta} f,v_i\big\rangle_\CH}\big\rvert \\
& \lesssim \big\|{\Th^{-\alpha} f}\big\|_\CH \kappa^{2(\alpha-\beta)+1}\tau^{-\beta},
\end{align*}
with $\beta:=\min(\alpha+1/2,\nu)$.
Therefore, the approximation rate increases by $1/2$ (qualification permitting).
Combining all together, we obtain the following bound in $\Ltwo$.
\begin{cor}
Assume that $g_\tau$ has qualification $ \nu \in (0,\infty] $,
$f\in\range(\Th^\alpha)$ for some $ \alpha > 0 $,
and $\lambda\mapsto \lambda g_\tau(\lambda)$ is Lipschitz continuous on $[0,\kappa^2]$ with Lipschitz constant $L(\tau)\lesssim \tau^p$, $ p \ge 1 $.
Let $\beta:=\min\{\alpha+1/2,\nu\}$
and set
$$
\tau := \lceil N^{\frac{1}{2(\beta+p)}}\rceil .
$$
Then, for every $ t > 0 $, with probability at least $1-2e^{-t}$ we have
\begin{equation*}
\big\|{f-\widehat{f}_{\tau,N}}\big\|_{\rho}\lesssim \big\|{\Th^{-\alpha} f}\big\|_\CH \LRP{\kappa^{2(\alpha-\beta)+1}+\kappa^{2\alpha+3}\sqrt{t}} N^{-\frac{\beta}{2(\beta+p)}}.
\end{equation*}
\end{cor}
See Table \ref{tab:list_of_error_rates} for specific rates regarding spectral functions from Table \ref{tab:list_of_gs}.

\paragraph{Monte Carlo wavelet approximation as noiseless kernel ridge regression}

We conclude this section with an observation that draws a link between Monte Carlo wavelets and regression analysis.
Let $ \widehat{f}_{\tau,N} $ be the Monte Carlo wavelet approximation \eqref{eq:hftauN} of $ f \in \CH $
at resolution $\tau$ given samples $ x_1,\dots,x_N $.
Then
\begin{align*}
 \widehat{f}_{\tau,N} = \sum_{j=0}^\tau \prefilter(\Themp)^2 \Themp f = g_\tau(\Themp) \Themp f .
\end{align*}
With the choice of the Tikhonov filter $ g_j(\la) = ( \la + \tau^{-1} )^{-1} $ (Table \ref{tab:list_of_gs}),
recalling \eqref{eq:kernel_matrix}, \eqref{eq:sampling} and \eqref{eq:out-of-sample},
and defining
$$
\Vy = [ f(x_1) , \dots , f(x_N) ]^\top, \qquad \Val = \Big(\VK + \tfrac{N}{\tau} \VI \Big)^{-1} \Vy ,
$$
we have
 \begin{align*}
 \widehat{f}_{\tau,N}
 & = \big(\Themp + \tfrac{1}{\tau} \Id_\CH \big)^{-1} \Themp f
 = \Big(\hInc^*\hInc + \tfrac{1}{\tau} \Id_\CH \Big)^{-1} \hInc^*\hInc f
 = \Big(\hInc^*\hInc + \tfrac{1}{\tau} \Id_\CH \Big)^{-1} \hInc^* \Vy \\
 & = \hInc^* \Big(\hInc\,\hInc^* + \tfrac{1}{\tau} \VI \Big)^{-1} \Vy
 = \frac{1}{N} \sum_{i=1}^N K(\cdot,x_i) \Big[ \Big(\tfrac{1}{N}\VK + \tfrac{1}{\tau} \VI \Big)^{-1} \Vy\Big][i] \\
 & = \sum_{i=1}^N K(\cdot,x_i) \Big[ \Big(\VK + \tfrac{N}{\tau} \VI \Big)^{-1} \Vy \Big][i] = \sum_{i=1}^N \Val[i] K(\cdot,x_i) .
\end{align*}
This is the (unique) solution to the kernel regularized least squares problem
\begin{align} \label{eq:RLS}
 \min_{\widehat{f} \in \CH} \frac{1}{N} \sum_{i=1}^N | y_i - \widehat{f}(x_i) |^2 + \la \| \widehat{f} \|_\CH^2 ,
\end{align}
where $ y_i = \Vy[i] $ and $ \la = \tau^{-1} $.
Therefore, $ \widehat{f}_{\tau,N} $ is the kernel ridge estimator for the noiseless regression problem
\begin{align*}
 y_i = f(x_i) \qquad i = 1,\dots,N ,
\end{align*}
and the squared reconstruction error $ \| f - \widehat{f}_{\tau,N} \|_\rho^2 $ is the generalization error of $\widehat{f}_{\tau,N}$.

Contrasting this with the optimal rate (in the minimax sense) for kernel ridge regression \cite{caponnetto-devito}
entails that the rate in Table \ref{tab:list_of_error_rates} is suboptimal for Tikhonov regularization,
and presumably for all other regularizers.
This is well expected from the crude Lipschitz bound used in Proposition \ref{prop:estimation_bound}.
The scope of the present work was to establish
a first result of convergence of randomly sampled frames, rather than identifying the optimality of the convergence rates.
Refinement of our bounds will be object of future investigation (see also Remark \ref{rmk:sharper}).

\section{Sobolev and Besov spaces in RKHS} \label{sec:Besov}

\providecommand{\embed}{{\mathrm{A}}}
\providecommand{\embedouter}{{\mathrm{B}}}
\providecommand{\soboleviso}{{\mathrm{U}}}

The convergence rates of the frame reconstruction error in Theorem \ref{thm:error}
depend on the approximation rates in Proposition \ref{prop:approx_bound},
hence on the regularity of the original signal $f$, as quantified by the condition $f\in\range(\Th^\alpha)$.
Thinking of $\Th$ as the inverse square root of the Laplacian
allows to interpret $\range(\Th^\alpha)$ as a Sobolev space.
The theory of smoothness function spaces \cite{triebel1992} plays a critical role in harmonic analysis,
and serves also as a base for the definition of statistical priors in learning theory \cite{binev2005}.
In this section we {examine} general notions of regularity and
their effect on the reconstruction error.
Many of the reported results on Besov spaces are well known \cite{fefupe16},
but we nonetheless include them here to be self contained and to adapt them to our setting and notation.
In particular, as already observed in Section \ref{sec:graphs},
it should be borne in mind that
the spectrum of the integral operator $\Th$ has inverse trend compared to that of a Laplace operator,
and therefore all the spectral definitions of the generalized Besov spaces must take this into account
in order to preserve the consistency with their classical counterparts.
As in the previous section, we assume $\supp (\rho) = \CX$.

\paragraph{Sobolev spaces as domains of powers of a positive operator}
By virtue of the spectral theorem, for every $\alpha>0$, $\Th^\alpha$ is a positive, bounded, injective operator on $\CH$,
with $ \sigma(\Th^\alpha) \subset (0,\kappa^{2\alpha}] $.
Thus,  $\Th^{-\alpha}$ is a positive, closed, densely-defined, injective operator with $ \sigma(\Th^{-\alpha}) \subset [\kappa^{-2\alpha},\infty) $.
We put the following
\begin{dfn}[Sobolev spaces]
For $ \alpha > 0 $, we define the \emph{Sobolev space} $\CH^\alpha$ by
$$
\CH^\alpha := \dom(\Th^{-\alpha}) = \range(\Th^\alpha) ,
$$
equipped with the norm
$$
\N{v}_{\CH^\alpha} := \N{\Th^{-\alpha} v}_\CH .
$$
\end{dfn}
$\Ch^\alpha$ is a Hilbert space. Moreover, we have
$$
 \CH^\alpha = \Big\{f \in \CH \,:\, \sum_{i\in\CI_\rho} \lambda_i^{-2\alpha} \SN{\EUSP{f}{v_i}_\CH}^2 < \infty\Big\} ,
$$
which expresses $\CH^\alpha$ in terms of the speed of decay of the Fourier coefficients,
thus generalizing the standard Sobolev spaces $ H^\alpha = W^{\alpha,2} $.
Theorem \ref{thm:error} establishes the convergence of Monte Carlo wavelets
for signals in the class $\CH^\alpha$.

\paragraph{Besov spaces as approximation spaces}
Besov spaces on Euclidean domains are traditionally defined by the decay of the modulus of continuity.
A characterization that is best suited to generalize to arbitrary domains,
and to which we also adhere,
is through approximation and interpolation spaces \cite{fefupe16,P81,triebel1992}.
We begin with the approximation perspective by defining a scale of Paley--Wiener spaces.

\begin{dfn}[Paley--Wiener spaces]
For $ \omega > 0 $, the \emph{Paley--Wiener} space $\PW(\omega)$ is defined by
\begin{equation*} 
 \PW(\omega) := \left\{ f \in \CH\,: \EUSP{f}{v_i}_\CH = 0 \text{ for }\lambda_i < \omega^{-1} \right\} = \overline{\spn} \left\{v_i : \lambda_i \ge \omega^{-1}\right\} .
\end{equation*}
The associated approximation error for $f\in\CH$ is
\[
 \CE(f,\omega) := \inf_{g\in\PW(\omega)} \N{f - g}_\CH = \N{\Pmat_{\PW(\omega)^\perp} f}_\CH = \Big({\sum_{\lambda_i<\omega^{-1}} \SN{\EUSP{f}{v_i}_\CH}^2}\Big)^{1/2} .
\]
\end{dfn}

The space $ \PW(\omega) $ is a closed subspace of $\CH$, and $ \bigcup_{w>0} \PW(\omega) $ is dense in $\CH$.
Note that $ \CE(f,\omega)\xrightarrow{\omega\to 0}\N{f}_\CH $ and $ \CE(f,\omega)\xrightarrow{\omega\to\infty}0$. 
Approximation spaces classify functions in $\CH$ according to the rate of decay of their approximation error.

\begin{dfn}[Besov  spaces]
For $ s > 0 $ and $ q \in [1,\infty) $, we define the \emph{Besov space} $ \CB^s_q $ as the approximation space
\begin{equation*}
  \CB^s_q:= \left\{ f \in \CH \,: \LRP{\int_0^\infty (\omega^{s} \CE(f,\omega))^q \frac{d\omega}{\omega}}^{1/q} < \infty \right\} ,
\end{equation*}
equipped with the norm
\begin{equation} \label{eqn:Bsq_condition}
 \N{f}_{\CB^s_q} := \N{f}_\CH + \LRP{\int_0^\infty (\omega^{s} \CE(f,\omega))^q \frac{d\omega}{\omega}}^{1/q} .
\end{equation}
The space $ \CB^s_\infty $ is defined with the usual adjustment.
\end{dfn}
Discretizing the integral in  \eqref{eqn:Bsq_condition}, we obtain the equivalent norm
\begin{equation} \label{eqn:Bsq_norm}
 \N{f}_\CH + \Big( \sum_{j\geq0} \left(2^{j s} \CE(f,2^j)\right)^q \Big)^{1/q} \asymp \N{f}_{\CB^s_q} .
\end{equation}
In particular, a function $ f \in \CB^s_q $ if and only if the sequence $ \left(2^{j s} \CE(f,2^j)\right)_{j\ge0} \in \ell^q $.
It is easy to see that the scale of spaces $\CB^s_q$ obeys the following lexicographical order \cite[Proposition 3]{P81}:
\begin{align}
 &\CB^{s}_q \supset \CB_p^{t} \quad \text{for } s < t , 
 \label{eqn:Bsq_inclusions_s} \\
 & \CB^s_q\subset \CB^s_p \quad \text{for } q < p .
 \nonumber
\end{align}

\paragraph{Besov spaces as interpolation spaces}
The Sobolev space $ \CH^\alpha$ is continuously embedded into $\CB^s_q $ for every $ \alpha > s $. 
Indeed, for $ f \in \CH^\alpha $ we have the Jackson-type inequality $\CE(f,\omega) \le \omega^{-\alpha} \| f \|_{\CH^\alpha} $, hence
\begin{align*}
 \sum_{j\geq0} (2^{j s} \CE(f,2^j))^q \le \N{f}_{\CH^\alpha}^q \sum_{j\geq0} 2^{-j q(\alpha-s)} < \infty .
\end{align*}
Furthermore, $\CB^s_q$ interpolates between $\CH^\alpha$ and $\CH$.
\begin{dfn}[interpolation spaces]
For quasi-normed spaces $\VE$ and $\VF$, $ \theta \in (0,1) $ and $q \in (0,\infty) $,
the quasi-normed \emph{interpolation space} $\LRP{\VE,\VF}_{\theta, q}$ is defined by
\[
\LRP{\VE,\VF}_{\theta, q} := \left\{ f\in\VE+\VF \,: \int_0^\infty \LRP{t^{-\theta} \CK(f,t)}^q \frac{dt}{t} <\infty \right\} ,
\]
 where $\CK(f,t)$ is Peetre's $K$-functional
\[ \CK(f,t) := \inf_{\substack{ f_0 + f_1=f \\ f_0 \in \VE , f_1 \in \VF }} \N{f_0}_{\VE} + t\N{f_1}_{\VF} . \]
The space $ \LRP{\VE,\VF}_{\theta, \infty} $ is defined with the usual adjustment.
\end{dfn}
Standard interpolation theory \cite{fefupe16,triebel1992} gives
\begin{equation} \label{eq:Bsq-inter}
 \CB^s_q = (\CH,\CH^\alpha)_{\frac{s}{\alpha},\,q} \quad \text{for } s \in (0,\alpha) \text{ and } q \in [1,\infty] ,
\end{equation}
with
\begin{equation} \label{eq:Bsq-K}
 \N{f}_{\CB^s_q} \asymp \N{f}_\CH + \left( \int_0^\infty \LRP{t^{-\theta} \CK(f,t)}^q \frac{dt}{t} \right)^{1/q} .
\end{equation}
In the next proposition we show that, as in the Euclidean setting,
the Besov space $\CB^s_2$ coincides with the Sobolev space $\CH^s$ of the same order.
{As in the classical setting, this is particular to the case $ q = 2 $.}
This is probably a known fact,
but we could find neither a proof nor a statement.
\begin{prop}
 For every $ s > 0 $, $ \CB^s_2 = \CH^s $ with equivalent norms.
\end{prop}
\begin{proof}
Let $ \al = 2 s $.
Then \eqref{eq:Bsq-inter} and \eqref{eq:Bsq-K} give $ \CB^s_2 = (\CH,\CH^{\al})_{\frac{s}{\al},\,2} = (\CH,\CH^{2s})_{\frac{1}{2},\,2} $ and
\begin{equation} \label{eq:K}
 \N{f}_{\CB^s_2}^2 \asymp \N{f}_\CH^2 + \int_0^\infty t^{-1} \CK(f,t)^2 \frac{dt}{t} . 
\end{equation}
Let $ \embed : \CH^\al \to \CH $ denote the canonical embedding $ \embed g = g $. Then, for $ f \in \CH $ and $ t > 0 $ we have
\begin{align} \label{eq:K-F}
 \CK(f,t)^2
 & = \inf_{\substack{ f_0 + \embed g=f \\ f_0 \in \CH , g \in \CH^\al }} (\N{f_0}_\CH + t \N{g}_{\CH^\al})^2 \nonumber \\
 & = \inf_{g \in \CH^\al} (\N{f - \embed g}_\CH + t \N{g}_{\CH^\al})^2
 \asymp \CG(f,t^2) ,
 \end{align}
with
 \[
 \CG(f,\la) := \inf_{g \in \CH^\al} \N{f - \embed g}_\CH^2 + \la \N{g}_{\CH^\al}^2 .
 \]
This infimum is attained by $ g = (\embed^*\embed + \la \Id_{\CH^\alpha})^{-1} \embed^*f $.
 Since
 \[
 (\embed^*\embed + \la \Id_{\CH^\alpha})^{-1} \embed^* = \embed^* (\embed\embed^* + \la \Id_\CH)^{-1} ,
 \]
 defining $ \embedouter := \embed\embed^* :\CH\rightarrow\CH$ we obtain
 \[
 \embed (\embed^*\embed + \la \Id_{\CH^\alpha})^{-1} \embed^* = \embedouter ( \embedouter + \la \Id_{\CH})^{-1} .
 \]
 Let $ \embed^* = \soboleviso (\embed\embed^*)^{1/2} = \soboleviso \embedouter^{1/2} $ be the polar decomposition of $\embed^*$, where $ \soboleviso : \CH \to \CH^\al $ is unitary. We have
 \[
 \CG(f,\la) = \N{(\Id_{\CH^\alpha} - \embedouter( \embedouter + \la \Id_{\CH^\alpha} )^{-1}) f}_\CH^2 + \la \| \soboleviso \embedouter^{1/2} (\embedouter + \la \Id_{\CH^\alpha})^{-1} f \|_{\CH^\al}^2.
 \]
Since
 \((\Id_{\CH} - \embedouter( \embedouter + \la \Id_{\CH} )^{-1} ) ( \embedouter + \la \Id_{\CH}) = \la \Id_{\CH},\)
 it follows that
\begin{align} \label{eq:F}
 \CG(f,\la) & = \la^2 \| (\embedouter + \la \Id_{\CH})^{-1} f \|_\CH^2 + \la \| \embedouter^{1/2}(\embedouter + \la \Id_{\CH})^{-1} f \|_\CH^2 \nonumber \\
 & = \la \left[ \la \langle (\embedouter + \la \Id_{\CH})^{-2} f , f \rangle_\CH + \langle \embedouter (\embedouter + \la \Id_{\CH})^{-2} f , f \rangle_\CH \right] \nonumber  \\
 & = \la \langle (\embedouter + \la \Id_{\CH})^{-2} (\la \Id_\CH + \embedouter) , f \rangle_\CH \nonumber \\
 & = \la \langle (\embedouter + \la \Id_{\CH})^{-1} f , f \rangle_\CH . 
\end{align}
Plugging \eqref{eq:K-F} and \eqref{eq:F} into \eqref{eq:K} we get
\begin{align*}
& \int_0^\infty t^{-1} \CK(f,t)^2 \frac{dt}{t} \asymp \int_0^\infty t^{-1} \CG(f,t^2) \frac{dt}{t} \\
 =& \int_0^\infty \EUSP{(\embedouter + t^2 \Id_{\CH^\alpha})^{-1} f}{f}_\CH dt
 = \int_0^\infty \int_0^\infty \frac{1}{\sigma + t^2} \EUSP{d\pi_{\embedouter}(\sigma) f}{f}dt ,
\end{align*}
where $ \pi_{\embedouter}$ is the spectral measure of ${\embedouter}$.
By Fubini we have
\begin{align*}
& \int_0^\infty \int_0^\infty \frac{1}{\sigma + t^2} dt \ \EUSP{d\pi_{\embedouter}(\sigma) f}{f} 
= \int_0^\infty \frac{1}{\sqrt{\sigma}} \arctan\Big(\frac{t}{\sqrt{\sigma}}\Big) \bigg|_0^\infty \EUSP{d\pi_{\embedouter}(\sigma) f}{f} \\
\asymp & \int_0^\infty \sigma^{-1/2} \EUSP{d\pi_{\embedouter}(\sigma) f}{f}
=\langle{\embedouter^{-1/2} f,f}\rangle_\CH= \big\|{\embedouter^{-1/4} f}\big\|_\CH^2 .
\end{align*}
Therefore, $ f \in \CB^s_2 $ if and only if $ f \in \dom(\embedouter^{-1/4}) $.
It now suffices to show $ \embedouter^{-1/4} = \Th^{-s} $, {whence $ \textN{\embedouter^{-1/4} f}_\CH^2 = \textN{f}_{\CH^s}^2 $.}
For any $ f \in \CH $ and $ g \in \CH^\al $ we have
\begin{align*}
\EUSP{f}{\embed g}_{\CH} = \EUSP{\embed^* f}{g}_{\CH^\al} = \EUSP{\Th^{-\al} \embed \embed^*f}{\Th^{-\al} \embed g}_\CH = \EUSP{\Th^{-2\al} \embedouter f}{g}_\CH .
\end{align*}
{Since $\CH^\al$ is dense in $\CH$, this implies} $ \Th^{-2\al} \embedouter = \Id_{\CH} $. 
Hence, $ \embedouter = \Th^{2\al} = \Th^{4s} $, which completes the proof.
\qed
\end{proof}

\paragraph{Besov spaces by wavelets coefficients}
The Besov norm can also be expressed by means of wavelet coefficients.
Let
$$
\filter(\lambda):=\sqrt{\lambda}\prefilter(\lambda) ,
$$
where $\prefilter$ is a filter as in Definition \ref{defn:filter}.
The partition of unity \eqref{eqn:part_unity} becomes
\begin{equation} \label{eq:part_unity-2}
\sum_{j\geq0}\filter(\lambda)^2=1 \quad \text{for all } \lambda\in(0,\kappa^2] .
\end{equation}
Moreover, in view of \eqref{eq:Fj(T)f}, for a frame $ \VPsi $ as in Definition \ref{defn:frames} we have
$$
\N{ \EUSP{f}{\psi_{j,\cdot}} }_{\Ltwo} = \N{\filter(\Th) f}_\CH ,
$$
and the frame property \eqref{eqn:frame_prop_H} can be rewritten as
\begin{equation} \label{eq:sumjFj(T)f}
 \N{f}_\CH^2 = \sum_{j\ge 0} \N{\filter(\Th) f}_\CH^2 .
\end{equation}
If we further assume the localization property (cf. Example \ref{ex:classical_filters})
\begin{equation} \label{eq:suppFj}
 \supp( F_0 ) \subset (2^{-1},\infty) , \qquad  \supp( F_j ) \subset ( 2^{-j-1} , 2^{-j+1} ) \quad \text{for } j \ge 1 ,
\end{equation}
a weighted $\ell^q$-norm of the sequence $ (\N{\filter(\Th) f}_\CH)_{j\ge 0} $ gives an equivalent characterization of the space $ \CB^s_q $.
 \begin{prop}[ {\cite[Theorem 3.18]{fefupe16}} ] \label{prop:norm_equivalence}
 Let $ \{\filter\}_{j\geq0} $ be a family of measurable functions $ \filter : [0,\infty) \to [0,\infty) $
 satisfying \eqref{eq:part_unity-2} and \eqref{eq:suppFj}.
 Then, for every $ f \in \CB^s_q$ we have
  \[
   \N{f}_{\CB^s_q} \asymp {\N{f}}_\CH + \Big( \sum_{j\geq0} \LRP{2^{j s} \N{\filter(\Th) f}_\CH }^q \Big)^{1/q}. 
   \]
 \end{prop}
 \begin{proof}
 We upper and lower bound the discretized norm in \eqref{eqn:Bsq_norm}.
Using \eqref{eq:sumjFj(T)f} (which holds thanks to \eqref{eq:part_unity-2}) and \eqref{eq:suppFj}, we have
  \begin{align*}
  & \CE(f,2^\ell)^2
= \N{\Pmat_{\PW(2^\ell)^\perp} f}_\CH^2 = \sum_{j\geq0}\N{\filter(\Th) \Pmat_{\PW(2^\ell)^\perp} f}_\CH^2 \\
 = & \sum_{j\geq0}\sum_{i\in\CI_\rho} \SN{\EUSP{\filter(\Th)\Pmat_{\PW(2^\ell)^\perp} f}{v_i}_\CH}^2
 = \sum_{j\geq0}\sum_{i\in\CI_\rho} \SN{\EUSP{\Pmat_{\PW(2^\ell)^\perp} f}{\filter(\Th) v_i}_\CH}^2 \\
 = & \sum_{j \ge 0} \sum_{ \substack{ \la_i < 2^{-\ell} \\ \la_i \in (2^{-j-1},2^{-j+1}) } } \hspace{-15pt} \filter(\lambda_i)^2 \SN{\EUSP{f}{v_i}_\CH}^2
 =  \sum_{j \geq \ell} \sum_{\la_i \in (2^{-j-1},2^{-j+1})} \filter(\lambda_i)^2 \SN{\EUSP{f}{v_i}_\CH}^2  \\
 = & \sum_{j\geq \ell} \sum_{i\in\CI_\rho}  \filter(\lambda_i)^2 \SN{\EUSP{f}{v_i}_\CH}^2
 = \sum_{j\geq \ell} \N{\filter(\Th) f}_\CH^2 .
  \end{align*}
Thus, by the discrete Hardy inequality (Lemma \ref{lem:Hardy_ineq}), we get
\begin{align*}
 \Big( \sum_{\ell\geq0} (2^{\ell s} \CE(f,2^\ell))^q \Big)^{1/q} & \le \Big( \sum_{\ell\geq0} \Big(2^{\ell s}  \sum_{j\ge \ell} \N{\filter(\Th) f}_\CH \Big)^q \Big)^{1/q} \\
 & \le C_{sq} \Big( \sum_{j\geq0}\left(2^{j s} \N{\filter(\Th) f}_\CH \right)^q \Big)^{1/q} ,
\end{align*}
with $C_{sq}=\frac{2^{sq}}{2^{sq}-1}$.
Conversely,
$\filter(\Th) g = 0$ for every $ g \in \PW(2^{j}) $, and therefore
$$
 \N{\filter(\Th) f}_\CH = \N{\filter(\Th) (f-g)}_\CH \le \N{\filter(\Th)}_\CH \N{f-g}_\CH \le \N{f-g}_\CH ,
$$
whence
\begin{equation*}
 \N{\filter(\Th) f}_\CH \le \inf_{g\in\PW(2^{j})} \N{f-g}_\CH = \CE(f,2^{j}) .
\end{equation*}
 \qed
\end{proof}

\paragraph{Convergence of spectrally-localized Monte Carlo wavelets}
Proposition \ref{prop:norm_equivalence} can be used to obtain
approximation bounds for frames
built with filters satisfying the localization property \eqref{eq:suppFj}.
\begin{prop} \label{prop:besov_approximation_bound}
Under the conditions of Proposition \ref{prop:norm_equivalence},
for every $ f \in \CB^s_q $ and $ \epsilon \in (0,s) $, we have
 \[
  \Big\|{\sum_{j>\tau} \Th_j f}\Big\|_\CH \lesssim \begin{cases}
   \N{f}_{\CB^s_q} 2^{-\tau s} & \text{for } q \in [1,2] \\
   \N{f}_{\CB^{s-\epsilon}_2} 2^{-\tau (s-\epsilon)} & \text{for } q \in (2,\infty]
  \end{cases} .
 \]
\end{prop}

\begin{proof}
By Proposition \ref{prop:norm_equivalence}, we have
 \begin{align*}
 \sum_{j>\tau} \N{\filter(\Th)f}_\CH^q = \sum_{j>\tau} 2^{-jsq} \left( 2^{j s} \N{\filter(\Th)f}_\CH \right)^q \lesssim 2^{-(\tau+1)s q} \N{f}_{\CB^s_q}^q .
\end{align*}
Also, \eqref{eq:part_unity-2} implies
$ \left\vert{\sum_{j} \filter(\lambda_i)^2}\right\vert^2 \le \sum_{j} \filter(\lambda_i)^2 $.
Hence, for $ q \le 2 $ we obtain
\begin{align*}
 \Big\|{\sum_{j>\tau} \Th_j f}\Big\|_\CH^2
 & = \sum_{i\in\CI_\rho} \Big\vert{\sum_{j>\tau} \filter(\lambda_i)^2}\Big\vert^2 \SN{\EUSP{f}{v_i}_\CH}^2
 \le \sum_{j>\tau}  \sum_{i\in\CI_\rho} \filter(\lambda_i)^2 \SN{\EUSP{f}{v_i}_\CH}^2 \\
 & = \sum_{j>\tau} \N{\filter(\Th) f}_\CH^2
 = \big\|{ \LRP{\N{F_{\tau+j}(\Th)f}}_{j\ge1} \big\| }_{\ell^2}^2 \\
 & \le \big\|{ \LRP{\N{F_{\tau+j}(\Th)f}}_{j\ge1} }\big\|_{\ell^q}^2
 \le \big({ 2^{-(\tau+1) s} \N{f}_{\CB^s_q} }\big)^2.
\end{align*}
If $q>2$, then $ \CB^s_q \subset \CB^{s-\epsilon}_2 $ for every $ \epsilon \in (0,s)$,
thanks to \eqref{eqn:Bsq_inclusions_s}, and the claim follows.
\qed
\end{proof}

Putting together Proposition \ref{prop:besov_approximation_bound} and Proposition \ref{prop:estimation_bound}  yields a convergence result for Monte Carlo wavelets with localized filters.
\begin{thm}\label{thm:Besov_cons}
Assume that $ \filter $ satisfies \eqref{eq:suppFj},
$ f \in \CB^s_q $ with $ q \in [1,2] $,
and $\lambda\mapsto \lambda g_\tau(\lambda)$ is Lipschitz continuous on $[0,\kappa^2]$ with Lipschitz constant $L(\tau)\lesssim 2^\tau$.
Set
$$
 \tau= \lceil{ {\tfrac{1}{2s+2}} \log_2 (N)}\rceil .
$$
Then, for every $ t > 0 $, with probability at least $1-2e^{-t}$ we have
\begin{equation*}
\big\|{f-\widehat{f}_{\tau,N}}\big\|_\CH \lesssim \N{f}_{\CB^s_q} \big({1+\kappa^{2}\sqrt{t}}\big) N^{-\frac{s}{2s+2}}.
\end{equation*}
\end{thm}

Compared to Theorem \ref{thm:error}, Theorem \ref{thm:Besov_cons} requires the resolution $\tau$ to grow only logarithmically with respect to the sample size $N$.
Note that the conditions of Theorem
\ref{thm:Besov_cons} exclude the spectral functions of Table \ref{tab:list_of_gs},
since they do not satisfy \eqref{eq:suppFj}.
Examples of admissible filters are given instead by Example \ref{ex:classical_filters},
which have local support \eqref{eq:suppFj} but exponential Lipschitz constant.

\section{Concluding remarks and future directions} \label{sec:conclusions}

We presented a construction of tight frames
which extends wavelets on general domains
based on spectral filtering of a reproducing kernel.
Depending on the measure considered, our construction leads to continuous or discrete frames,
covering non-Euclidean structures such as Riemannian manifolds and weighted graphs.
Besides standard frequency-localized filters commonly used in wavelet frames,
we defined admissible spectral filters resorting to methods from regularization theory,
such as Tikhonov regularization and Landweber iteration.
Regarding discrete measures as empirical measures arising from independent realizations of a continuous density,
we interpreted discrete frames as Monte Carlo estimates of continuous frames.
We proved that the Monte Carlo frame converges to the corresponding deterministic continuous frame,
and provided finite-sample bounds in high probability,
with rates that depend on the Sobolev or Besov class of the reproduced signal.
This demonstrates the stability of empirical frames built on sampled data.

In future work we intend to study the numerical implementation of our Monte Carlo wavelets,
along with possible applications in graph signal processing, regression analysis and denoising.
Further theoretical investigation may include $L^p$ Banach frame extensions,
sparse representations,
nonlinear approximation rates, Lipschitz bound refinements,
and explicit localization properties for specific families of kernels.

\begin{appendices}

\section{Appendix}

We recall the following result, whose proof can be collected from \cite{f2005abstract}.
\begin{lem}\label{dense}
Let $(\Omega;\mu)$ be a measure space and $\hh$ a Hilbert
space. Given a weakly measurable mapping $\omega\mapsto \Psi_\omega$ from
$\Omega$ to $\hh$, assume there exists a dense subset $\mathcal
D\subset \hh$, and a constant $C>0$, such that, for every $f\in  \mathcal D$,
\begin{equation}
  \label{eq:30}
  \int_{\Omega} \SN{\EUSP{f}{\Psi_\omega}_\hh}^2 d\mu(\omega) \leq C \|{f\|}^2 .
\end{equation}
Then~\eqref{eq:30} holds for every $f\in \hh$.
Furthermore, there exists a positive bounded operator $\Amat:\CH\rightarrow\CH$ such that,
for every $ f , g \in \hh $,
\begin{equation*}
  \EUSP{\Amat f}{g}_\CH = \int_{\Omega} \EUSP{f}{\Psi_\omega}_\hh
  \EUSP{\Psi_\omega}{g}_\hh d\mu(\omega) .
\end{equation*}
\end{lem}
\begin{proof}
For $f\in\CH$, define the measurable mapping
\[
\VV f:\Omega\to\bbC  \qquad \VV f(\omega) := \EUSP{f}{\Psi_\omega}_{\CH} .
\]
Let $ \CS :=\{ f\in\CH : \VV f\in L^2(\Omega;\mu) \} $.
The subspace $\CS$ is dense in $\CH$ since $ \CS \supset \CD $,
and the operator $\VV:\CS\to L^2(\Omega;\mu)$ is closed. Indeed, fix a sequence $(f_n) \subset \CS$ converging to $f\in\CH$ and such
that $(\VV f_n)$ converges to $F\in L^2(\Omega;\mu)$. 
Then, possibly passing to a subsequence, there is a subset $E\subset \Omega$ of measure zero such that, for all
$\omega\not\in E$,
\[
F(\omega)=\lim_{n\to\infty} \VV f_n(\omega) =
\lim_{n\to\infty}\EUSP{f_n}{\Psi_\omega}_{\CH} = \EUSP{f}{\Psi_\omega}_{\CH}.
\]
Then $f\in \CS$ and $F=\VV f$.  Moreover,
\[
\N{\VV f}^2_{L^2(\Omega;\mu)} = \lim_{n\to \infty} \N{\VV f_n}_{L^2(\Omega;\mu)}^2 \leq C \lim_{n\to
  \infty} \N{f_n}^2_{\CH} = C \N{f}_\CH^2 .
\]
Thus, $\VV$ is a bounded operator, and the closed graph theorem implies $\CS=\CH$, {\it i.e.} \eqref{eq:30} holds for all $f\in\CH$.
The second statement follows by defining $ \Amat := \VV^*\VV $.
\qed
\end{proof}

The simple proof of the following bound is due to A. Maurer.
\begin{lem} \label{lem:Lip}
 Let  $ \mathrm{A} , \mathrm{B} $ be self-adjoint operators on a separable Hilbert space $\CH$,
 and let $ F : \bbR \to \bbC $ be a Lipschitz continuous function with Lipschitz constant $L$.
Then 
 $$
 \| F(\mathrm{A}) - F(\mathrm{B}) \|_\HS \le L \| \mathrm{A} - \mathrm{B} \|_\HS .
 $$
\end{lem}
\begin{proof}
 Let $ \{e_i\}_{i\in \CI} $ and $ \{f_j\}_{j\in \CJ} $ be orthonormal bases of $\CH$
 such that $ \mathrm{A}e_i = \la_i e_i $ and $ \mathrm{B}f_j = \mu_j f_j $.
 Then
 \begin{align*}
  \| F(\mathrm{A}) - F(\mathrm{B}) \|_\HS^2
  & = \sum_{i\in \CI , j\in \CJ} | \langle (F(\mathrm{A}) - F(\mathrm{B})) e_i , f_j \rangle_\CH |^2 \\
  & = \sum_{i\in \CI , j\in \CJ} | F(\la_i) - F(\mu_j) |^2 | \langle e_i , f_j \rangle_\CH |^2 \\
  & \le L^2 \sum_{i\in \CI , j\in \CJ} | \la_i - \mu_j |^2 | \langle e_i , f_j \rangle_\CH |^2
  = L^2 \| \mathrm{A} - \mathrm{B} \|_\HS .
 \end{align*}
 \qed
\end{proof}

We include a proof of the discrete Hardy inequality 
  \cite[equation 5.2]{DP88} where we explicitly compute the Hardy constant.
\begin{lem}[Hardy inequality]\label{lem:Hardy_ineq}
Let $\LRP{b_j}_{j\geq0}$ and $\LRP{a_j}_{j\geq0}$ be two sequences such that 
\[
\SN{b_j}\leq \Big({\sum_{k\geq j} \SN{a_k}^p}\Big)^{1/p} \quad \text{for } 0<p\leq q .
\]
Then, for every $ s > 0 $, we have
\[ \sum_{j\geq0} \LRP{ 2^{js} \SN{b_j}}^q \leq \frac{2^{sq}}{{2^{sq}-1}}\sum_{j\geq0} \LRP{2^{js} \SN{a_j}}^q,\]
provided all the sums are finite.
\end{lem}
\begin{proof}
Let $\alpha = \frac{q}{p}$, and let $\beta$ be such that $sp > \beta >0.$
Since $ p \le q $, we have $\N{\cdot}_{\ell_q}\leq \N{\cdot}_{\ell_p}$, hence
\begin{align*}
& \sum_{j\geq0} \LRP{ 2^{j s} \SN{b_j}}^q
\le \sum_{j\geq0} 2^{jsq}\N{\LRP{a_{j+k}}_{k\geq0}}_{\ell_q}^q \\
\le & \sum_{j\geq0} 2^{jsq}\N{\LRP{a_{j+k}}_{k\geq0}}_{\ell_p}^q
= \sum_{j\geq0} 2^{jsq} \Big({\sum_{k\geq j} 2^{-k\beta} 2^{k\beta} \SN{a_k}^p}\Big)^\alpha .
\end{align*}
Assume now $\alpha\in(1,\infty)$. Applying the H\"older inequality with $1/\alpha + 1/\alpha'=1$ we have
\begin{align*}
\sum_{k\geq j} \LRP{2^{-k\beta}} \LRP{2^{k\beta} \SN{a_k}^p}
& \le \Big({\sum_{k\geq j} 2^{-k\beta\alpha'}}\Big)^{1/\alpha'}\Big({\sum_{k\geq j} 2^{k\beta\alpha} \SN{a_k}^{\alpha p}}\Big)^{1/\alpha} \\
& =\frac{2^{\beta}}{\LRP{2^{\beta\alpha'}-1}^{1/\alpha'}} 2^{-j\beta}\Big({\sum_{k\geq j} 2^{k\beta\alpha} \SN{a_k}^{\alpha p}}\Big)^{1/\alpha} . \end{align*}
Plugging this in and using $\alpha p = q$ we get
\begin{align*}
\sum_{j\geq0} \LRP{ 2^{j s} \SN{b_j}}^q
& \le C_1 \sum_{j\geq0} 2^{jsq} 2^{-j\beta\alpha}\Big({\sum_{k\geq j} 2^{k\beta\alpha} \SN{a_k}^{\alpha p}}\Big) \\
& = C_1 \sum_{j\geq0} 2^{j(sq-\beta\alpha)}\Big({\sum_{k\geq j} 2^{k\beta\alpha} \SN{a_k}^{q}}\Big),
\end{align*}
with
\[ C_1 :=\frac{2^{\alpha\beta}}{\LRP{2^{\beta\alpha'}-1}^{\alpha/\alpha'}}.\]
Changing the order of summation we get
\begin{align*}
& \sum_{j\geq0} \LRP{ 2^{j s} \SN{b_j}}^q
\le C_1 \sum_{j\geq0} 2^{j\beta\alpha} \SN{a_j}^q \sum_{k\le j} 2^{k\LRP{sq-\beta\alpha}} \\
\le & \ C_1 C_2\sum_{j\geq0} 2^{jsq} \SN{a_j}^q
= C_1 C_2\sum_{j\geq0}\LRP{2^{js} \SN{a_j}}^q ,
\end{align*}
with
$$
C_2 := \frac{2^{sq-\beta\alpha}}{2^{sq-\beta\alpha}-1} ,
$$
since 
\[ \sum_{k\le j} 2^{k\LRP{sq-\beta\alpha}} = \frac{1}{2^{sq-\beta\alpha}-1} \LRP{2^{(j+1)(sq-\beta\alpha)} -1} \leq \frac{2^{sq-\beta\alpha}}{2^{sq-\beta\alpha}-1} {2^{j(sq-\beta\alpha)}}.\]
We have
\[ C_1C_2 = \frac{2^{sq}}{\LRP{2^{\beta\alpha'}-1}^{\alpha/\alpha'}\LRP{2^{sq-\beta\alpha}-1}}.\]
If $\alpha=1$ ($p=q$), then $\alpha'=\infty$, in which case $C_1=1$ and therefore
\[ C_1C_2 = \frac{2^{sq-\beta}}{{2^{sq-\beta}-1}} \]
for all $ \beta \in (0,sq) $. Thus, we may set $C:=\frac{2^{sq}}{{2^{sq}-1}}$.
\qed
\end{proof}

\end{appendices}

\section*{Acknowledgements}
Part of this work has been carried out at the Machine Learning Genoa (MaLGa) center, Universit\`a di Genova (IT).
Ernesto De Vito is part of the Machine Learning Genoa Center (MalGa) and he is a member of the Gruppo Nazionale per l’Analisi Matematica, la Probabilit\`a e le loro Applicazioni (GNAMPA) of the Istituto Nazionale di Alta Matematica (INdAM).
LR and SV acknowledge the financial support of the European Research Council (grant SLING 819789), the AFOSR projects FA9550-17-1-0390  and BAA-AFRL-AFOSR-2016-0007 (European Office of Aerospace Research and Development), and the EU H2020-MSCA-RISE project NoMADS - DLV-777826.
VN and ZK acknowledge the support from RCN-funded FunDaHD project No 251149/O70.

\section*{References}

\bibliography{biblio}

\end{document}

%% file: macros.tex
\providecommand*{\Span}{\operatorname{span}}     

\providecommand{\ker}{\operatorname{ker}}

\providecommand*{\Span}[1]{\operatorname{Span}\left\{{#1}\right\}}     
\providecommand{\Supp}{\operatorname{supp}}                            
\providecommand{\supp}{\Supp}

\providecommand{\range}{\operatorname{range}}                

\providecommand{\Id}{\Op{Id}}                     

\providecommand{\diag}{\operatorname{diag}}

\newcommand{\Vf}{{\mathbf{f}}}
\newcommand{\Vg}{{\mathbf{g}}}

\newcommand{\Vu}{{\mathbf{u}}}
\newcommand{\Vv}{{\mathbf{v}}}

\newcommand{\Vy}{{\mathbf{y}}}

\newcommand{\Val}{{\bm{\alpha}}}

\newcommand{\VD}{{\mathbf{D}}}
\newcommand{\VE}{{\mathbf{E}}}
\newcommand{\VF}{{\mathbf{F}}}

\newcommand{\VI}{{\mathbf{I}}}

\newcommand{\VK}{{\mathbf{K}}}
\newcommand{\VL}{{\mathbf{L}}}
\newcommand{\VM}{{\mathbf{M}}}

\newcommand{\VU}{{\mathbf{U}}}
\newcommand{\VV}{{\mathbf{V}}}
\newcommand{\VW}{{\mathbf{W}}}

\newcommand{\VPsi}    {\mathbf{\Psi}}

\newcommand{\VPhi}    {\mathbf{\Phi}}

\newcommand{\Amat}{{\rm A}}

\newcommand{\Lmat}{{\rm L}}

\newcommand{\Pmat}{{\rm P}}

\newcommand{\Umat}{{\rm U}}

\providecommand{\Ch}{{\cal H}}

\providecommand{\CA}{{\cal A}}
\providecommand{\CB}{{\cal B}}

\providecommand{\CD}{{\cal D}}
\providecommand{\CE}{{\cal E}}

\providecommand{\CG}{{\cal G}}
\providecommand{\CH}{{\cal H}}
\providecommand{\CI}{{\cal I}}
\providecommand{\CJ}{{\cal J}}
\providecommand{\CK}{{\cal K}}

\providecommand{\CO}{{\cal O}}

\providecommand{\CS}{{\cal S}}

\providecommand{\CX}{{\cal X}}

\providecommand{\bbC}{\mathbb{C}}

\providecommand{\bbR}{\mathbb{R}}

\newcommand*{\EUSP}[2]{\left<{#1},{#2}\right>} 

\providecommand*{\N}[1]{\left\|{#1}\right\|} 
\newcommand*{\SN}[1]{\left|{#1}\right|}      

\newcommand*{\LRP}[1]{\left(#1\right)}

\newcommand*{\Op}[1]{\mathsf{#1}} 

\providecommand*{\Lp}[2][\defaultdomain]{L^{#2}({#1})}

\newcommand*{\Ltwo}[1][\defaultdomain]{\Lp[#1]{2}}

